\documentclass[10pt]{article}

\usepackage{amsfonts}
\usepackage{amsmath}
\usepackage{amssymb}
\usepackage{amsthm}
\usepackage{gensymb}
\usepackage{graphicx}
\usepackage[section]{placeins}
\usepackage{float}
\usepackage{cite}
\usepackage{setspace}
\usepackage{mathrsfs}
\usepackage{color}
\usepackage{soul}
\usepackage{comment}
\usepackage{booktabs}

\newtheorem{theorem}{Theorem}
\newtheorem{corollary}{Corollary}

\newtheorem{prop}{Proposition}
\newtheorem{definition}{Definition}

\newcommand{\R}{\mathbb R}

\newcommand{\T}{^\mathsf{T}}

\newcommand{\mL}{\mathcal L}

\newcommand{\x}{\mathbf{x}}

\newcommand{\myC}{\gamma}
\def\d{\mathrm{d}}

\usepackage{geometry}
\author{Toby Sanders, Justin J. Konkle, Erica E. Mason, Patrick W. Goodwill}

\title{Multi-Harmonic Gridded 3D Deconvolution (MH3D) for Robust and Accurate Image Reconstruction in MPI for Single Axis Drive Field Scanners}
\date{ }

\begin{document}
\maketitle

\begin{abstract}
\textbf{Objective:} This work introduces a new magnetic particle imaging (MPI) reconstruction framework based on multi-harmonic 3D deconvolution (MH3D) of gridded portraits, offering a principled, model-driven approach to MPI imaging.

\textbf{Approach:} MH3D defines a convolutional forward model using higher harmonic portraits, which are gridded images formed from filtered frequency-domain signal components. Each harmonic portrait is modeled as a convolution with a distinct PSF, closely approximated by derivatives of the Langevin function, and incorporates receive sensitivity and mesh downsampling for accurate modeling. We also introduce practical strategies for calibration, phase correction, and artifact reduction.

\textbf{Main Results:} We validate the MH3D approach using analytic approximations, numerical simulations, and experimental phantom data. {\color{black}MH3D yields high-resolution 3D reconstructions on seconds-scale runtimes, improves image quality relative to common 3rd-harmonic-only reconstructions, and achieves image quality and resolution comparable to a generalized model-based method in simulations and phantom experiments.}

\textbf{Significance:} {\color{black}This work offers new theoretical insight into MPI signal structure, unveiling the methodological and theoretical underpinnings absent in earlier single-harmonic or heuristic methods, thereby supporting accurate and robust 3D imaging with excellent computational efficiency.}
\end{abstract}




\section{Introduction}
Magnetic particle imaging (MPI) is an emerging medical imaging modality that can map superparamagnetic iron oxide nanoparticle (SPION) concentrations in the human body. It offers a unique set of imaging capabilities complementary to traditional imaging modalities like MRI, CT scans, and ultrasound, and can produce linear, quantifiable, high-resolution images of the SPION tracers \cite{gleich2005tomographic,dulinska2019superparamagnetic,panagiotopoulos2015magnetic,knopp2012magnetic,saritas2013magnetic}. The potential medical applications of MPI are currently a very active area of research and include but are not limited to cardiovascular (perfusion) imaging, cell tracking, and sentinel lymph node imaging \cite{zheng2015magnetic,sehl2020perspective,weizenecker2009three, panagiotopoulos2015magnetic, bulte2015quantitative,talebloo2020magnetic}. The list of applications will continue to evolve and expand as the MPI technology transitions from the preclinical research space and into clinical labs and hospitals.

The work in this article focuses on a fundamental component of MPI, the image reconstruction problem, the process by which the raw signal data acquired from the scan is processed to form images representing the SPION maps. Our new method develops a complete model utilizing the full breadth of information in the MPI signal and is based on deconvolution of \textit{harmonic portrait} data. The advantage of this new model is in the intermediate harmonic portrait domain, where we can analyze the data, correct potential artifacts, and tune and select parameters based on the data analysis.

An MPI scanner works by selectively exciting magnetic particles with an oscillating drive field. Receive coils then detect the change in magnetization of the particles induced by the drive field. The received time-domain signal data is then used to reconstruct the SPION density map. The unique property of the SPIONs is that their magnetization response to the magnetic field is given by the nonlinear Langevin function, which allows for harmonic signals necessary for MPI.

The static gradient magnetic field, $H(\vec x)$, creates a field-free region (FFR), which is either a field-free point (FFP) or field-free line (FFL) in 3D. The drive field causes the FFR to oscillate and the particles near the FFR to selectively experience a change in magnetization, while the particles far from the FFR remain in a saturated magnetic state as characterized by the Langevin function. This effect creates spatially localized information about the SPION density in the received signal, which allows for accurate image reconstruction when paired with FFR location information. 

\subsection{Contribution}

This work develops a new model-based method for image reconstruction in MPI, which is rooted in characterizing the signals at their harmonics. The harmonics are located at integer multiples of the drive frequency, $k \cdot f_0$, where $f_0$ is the sinusoidal drive field frequency and $k\ge2$ is an integer\footnote{The fundamental harmonic at $f_0$ is absent in the MPI receive signal due to the necessary receive chain filter of the excitation drive field.}. Empirically, the signal spreads to a specific bandwidth around each harmonic defined by the shift field frequency, gradient, and particle characteristics.

We refer to our method as multi-harmonic 3D deconvolution (MH3D). The term \textit{multi-harmonic} refers to the idea that we first compress the signal into a harmonic domain by interpolating the filtered harmonic data onto the FFP locations. This compressed data form we refer to as \textit{harmonic portraits}, since they represent a portrait-like representation of the SPIONs, with unique \textit{signatures} at each harmonic. These portraits are then used to compute a multi-frame 3D deconvolution to form the image reconstruction. Each of the portraits could in principle be used individually as a single-frame deconvolution, but using multiple harmonics together improves the resolution both theoretically and empirically.

In contrast, prior work developed a complete physics-based computational image reconstruction model for MPI\cite{sanders2025physics}, which is amenable to arbitrary scanning parameters and geometries. However, the current method is also a physics-based image reconstruction model that does not trivially generalize to arbitrary scanning parameters. It offers however the distinct advantage of operating in an intermediate portrait domain, which allows for careful data analysis and complementary data processing prior to image reconstruction  (see some of the methods outlined in Section \ref{sec: additional}). For example, we can detect and filter background signal, analyze which harmonics to use, select reconstruction parameters, and implement very simple data calibration methods that have so far been less trivial for the former generalized model. This portrait domain analysis also allows for detection and debugging of potential scanner hardware issues and allows for quick \textit{intermediate} image visuals during a scan instead of waiting for the full reconstruction. For these reasons, MH3D has been an important tool for progress in our group's work for both software and hardware development, while the former model will also continue to be an important tool in the future of MPI as the technology continues to progress. 

{\color{black} 
The robustness of the MH3D methodology on our next-generation scanner has been critical to our progress, consistently providing quality reconstructions with less time spent optimizing when compared with methods such as X-space and the generalized model, while also avoiding tedious system calibrations required for system-matrix approaches\cite{knopp2012magnetic}. This improved stability and resolution has enabled, for the first time, the possibility of human subject testing on our large field-of-view scanner --- an important milestone for MPI as it advances toward applications such as perfusion imaging\cite{guo2025exploring} and cell tracking \cite{sehl2020perspective}.
}

Another important result of MH3D is the theoretical analysis of the mathematical components of MPI that fall out of the methodology. We present many known concepts in MPI from this new perspective, and also some lesser-known concepts, such as which harmonics are most important in MPI and how many are needed for image reconstruction. This work primarily focuses on the important algorithmic components of MH3D, while some of these theoretical components are also provided to support the methodology and understanding.

The concept for MH3D is applicable for a somewhat general set of scanner geometries. However, we primarily focus the mathematical descriptions from the viewpoint of our scanner's hardware to keep the details concise and aligned. Section \ref{sec: scanner} outlines the scanner parameters, and most of the descriptions thereafter assume this particular set up. The major concepts behind MH3D are provided in Section \ref{sec: main}. Additional but fundamental details behind MH3D are given in Section \ref{sec: additional}, while several examples are provided along the way to support the ideas presented.

\subsection{Related Previous Approaches in MPI}
While gridding harmonic data into a \textit{portrait-like} domain is not new in MPI, previous approaches vary in how this data is used for image reconstruction. This work provides a formal mathematical framework, along with numerical methods and calibration strategies, to robustly reconstruct images from harmonic portraits.

An earlier approach to harmonic deconvolution was proposed by one of the present authors using a single harmonic\cite{goodwill2010narrowband}. That method applied least-squares reconstruction, which is inevitably dominated by noise and artifacts. A variant involved clipping high-frequency components in $k$-space to reduce noise, but this introduces Gibbs ringing artifacts. These limitations ultimately led the author to abandon the approach in favor of what became the widely adopted X-space methods \cite{goodwill2010x}. Similarly, several studies have relied solely on the 3rd harmonic signal \cite{mason2022side, janssen2022single, nomura2024development, mcDonough2024tomographic}. These portraits approximate SPION distributions only near the drive field's DC bias and do not generalize to 3D imaging. For example, \cite{janssen2022single} identified the convolutional relationship between the image and harmonic PSF but used a system matrix-based reconstruction instead. The method in \cite{mason2022side} used 2D tomographic projection-based reconstruction from 3rd harmonic portraits, tailored to specific temporal resolution constraints. Our work generalizes this concept to 3D and introduces a more complete modeling and computational framework.

Multi-harmonic methods have also been explored. In \cite{liu2022weighted}, harmonic portraits similar to ours were defined but combined only via a weighted sum, which lacks theoretical rigor and limits reconstruction accuracy. {\color{black}Methods such as these and third-harmonic approaches are largely empirical, based on visual assessment of gridded harmonics. The common preference for the third harmonic is essentially incidental: the third derivative of the Langevin function most closely resembles a symmetric impulse response, as our theoretical development will make clear.} In contrast, our approach treats each harmonic portrait as a convolution with a distinct PSF combined with a receive-coil weighting, enabling principled multi-frame reconstruction.


{\color{black}
Finally, it is important to position MH3D relative to the two major classes of MPI reconstruction methods: system matrix (SM) approaches and X-space methods. SM-based methods \cite{knopp2012magnetic} are widely used due to their generality and strong empirical performance, but they require extensive calibration measurements, which becomes increasingly burdensome and must often be repeated for different tracer types or scanner conditions. The calibration matrix size ultimately limits image resolution and slows progress toward clinical translation. X-space methods \cite{goodwill2010narrowband,goodwill2011multidimensional} rely on a sequence of engineered processing steps, such as calibration, filtering, and gridding, rather than an explicit signal model.  As a result, X-space reconstructions must be customized for each scanner geometry, are more challenging to generalize to arbitrary scanner parameters (e.g., receive sensitivity, FFR trajectories), and cannot easily incorporate a priori constraints such as non-negativity or smoothness. This limits its robustness and flexibility compared to model-based approaches, highlighting the motivation for our adoption of MH3D on our current scanner configuration.

Significant progress has also been made in model-based reconstruction strategies. Knopp and colleagues have developed physics-based forward models that incorporate relaxation effects, field imperfections, and scanner-specific parameters, allowing reconstructions without full system matrix calibration and with improved quantitative accuracy \cite{knopp2011model,knopp2017relaxation,knopp2021advances, droigk2025chebyshev, maass2024equilibrium}. While these approaches represent an important advance, they remain computationally demanding and are not generally practical for high-resolution 3D imaging.

Our MH3D framework builds on the same spirit of model-driven reconstruction in our previous work \cite{sanders2025physics}, but with the distinct innovation of operating in the harmonic portrait domain. This formulation retains the benefits of physics-based modeling while offering computational efficiency, reduced calibration requirements, and new opportunities for artifact suppression and intermediate data analysis, making it well-suited for large-scale 3D MPI.
}

\section{FFP Scanning Geometry} \label{sec: scanner}

Our scanner operates using an FFP with a static linear magnetic gradient field given by
\begin{equation}
	H(\vec x ) = G \vec x = G_0 \cdot 
	\begin{bmatrix}
		\frac{x}{2},  & \frac{y}{2}, &  z
	\end{bmatrix}\T ,
\end{equation}
where $G_0 = 0.554$ T/m. The transmit drive coil excites along the $z$-axis, which defines the drive field given by
\begin{equation}
	H_D(t) = \begin{bmatrix}
		0 , & 0, & B_{ex} \sin(2\pi f_0 t)
	\end{bmatrix}\T ,
\end{equation}
where $B_{ex}$ is the excitation amplitude. The position excursion in meters of the FFR due to the drive field is $A = B_{ex}/G$. The slow shift field, sometimes called the \textit{focus field}, rasters in a zigzag pattern in an $xy$ plane for a fixed $z$ location. We refer to this scan at a fixed $z$ position as a $z$-slab. Figure \ref{fig: portraits} (c) demonstrates our FFP scanning trajectory for a single $z$-slab.

\begin{figure}
	\centering
	\includegraphics[width=1\textwidth]{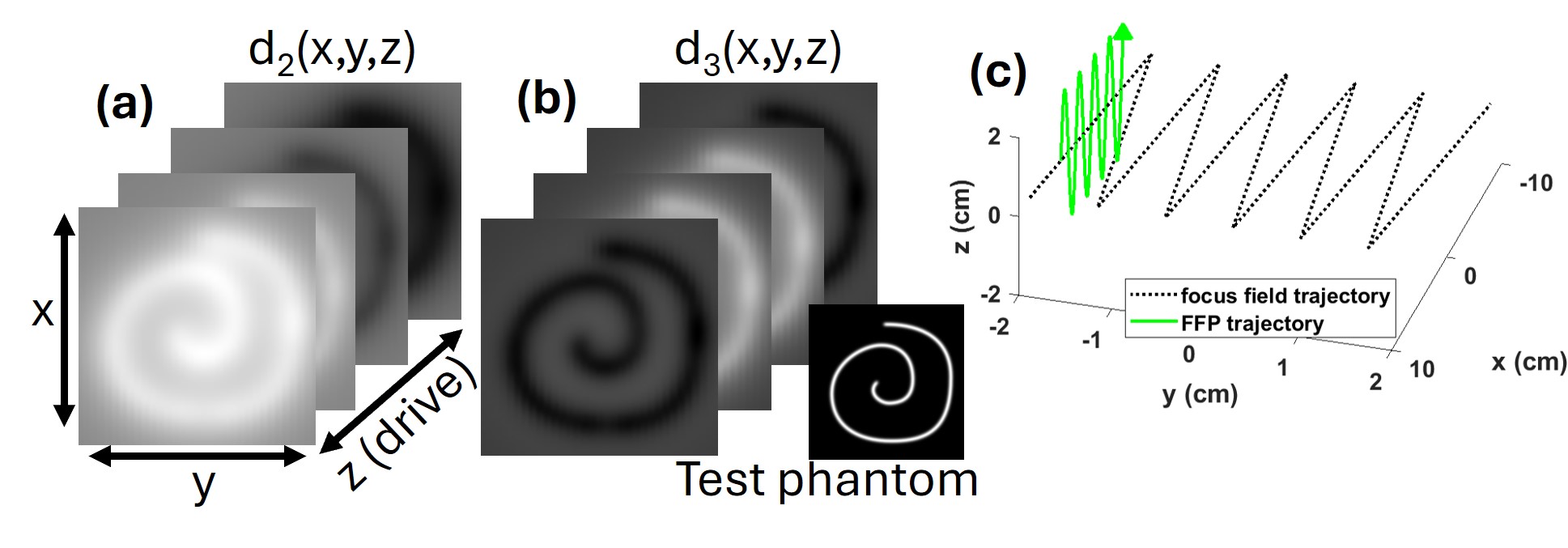}
	\caption{(a) and (b) show the 2nd and 3rd harmonic portraits at 4 consecutive $z$-slabs for a spiral phantom (shown in the bottom right of (b)) located in a plane at $z=0$. (c) demonstrates the FFP scanning trajectory for a single z-slab.}
	\label{fig: portraits}
\end{figure}

Multiple $z$ positions are acquired to obtain a full 3D FFP data acquisition for image reconstruction. The $z$-slab positions are $\{ z_j \}_{j=1}^{N_z}$, where $N_z$ is the total number of slabs. The positions are monotonically increasing, $z_1 < z_2 < \dots <z_{N_z}$, and equally spaced with $\Delta z = z_{j+1} - z_j $. Furthermore, $\Delta z$ is chosen so that the FFP trajectories between neighboring slabs overlap, which can be expressed as $\Delta z < 2 A$.

The received time-domain MPI signal for a single $z$-slab is given in \cite{goodwill2011multidimensional} by
\begin{equation}
	s(t ; \, z_j) =  m\frac{d}{dt} \iiint \rho(\vec x) \vec b_1 (\vec x)\T \frac{H_j(\vec x , t) }{\| H_j(\vec x , t) \| } \mathcal L (\beta \| H_j(\vec x , t) \|) \, d \vec x ,
\end{equation}
where
\begin{itemize}
	\item $\rho$ is the SPION density to image.
	\item  $H_j(\vec x , t)$ is the total magnetic field for the scanned position $z_j$ at time $t$.
	\item $\mathcal L$ is the Langevin function given by $\mathcal L(x) = \coth(x) - 1/x$.
	\item $\beta$ is a conglomerate SPION and scanner dependent constant, given by $\beta = \frac{\mu_0 m}{\kappa_B T}$, where $T$ is temperature, $m$ is the SPION specific magnetic moment, $\kappa_B$ is the Boltzmann constant, and $\mu_0$ is the vacuum permeability. 
	\item $\vec b_1$ is the vector field of spatially varying receive coil sensitivity values.
\end{itemize} 

We denote the FFP location for each position $z_j$ at time $t$ by $\vec \xi_j(t)$, and it is formally defined by the point satisfying $H_j(\vec \xi_j (t) , t) = 0$. 
For each fixed slab, the rastering pattern is such that it densely covers the entire $xy$ imaging FOV, but without repeating over a position in that plane. This makes for what is essentially a 1-to-1 relationship in our discretized domains between time and FFP location.

{\color{black}An introductory experimental imaging example with the scanner geometry just described is shown in Figure \ref{fig: LSI}, where the reconstruction used the MH3D method laid out in the sections that follow. This phantom contains an array of small vials at regularly spaced intervals, each containing 80 $\mu$g of VivoTrax (Magnetic Insight, Inc., Alameda, CA.). It is useful for demonstrating the shift invariance of both the scanner and corresponding reconstruction. Shown in Figure \ref{fig: LSI}(b)  are cross-section images of the reconstruction from two different views. Observe that these images demonstrate consistency in the reconstructed SPION intensities across the entire imaging FOV, although with some variation in the shapes due to mild scanner inhomogeneities. In addition, there is approximately a double resolution in the $z$-axis compared with the $x$ and $y$ axes due to the $z$-axis-only transmit.}

\begin{figure}
	\centering
	\includegraphics[width=0.6\textwidth]{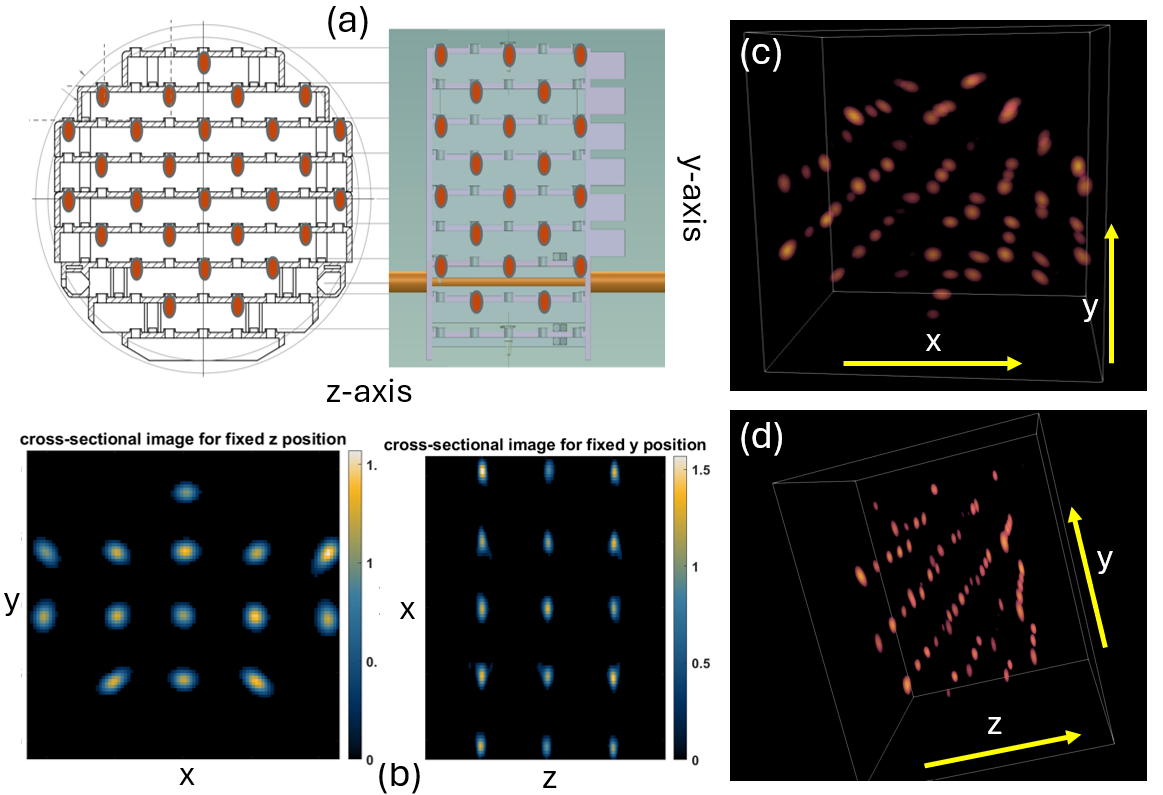}
	\caption{MH3D reconstructions of an experimental data set with a phantom containing an array of small vials of equal SPION volume (80 $\mu$g per vial) across the entire FOV, which is intended to test the linearity and shift invariance of the MPI scanner and reconstruction. (a) CAD model of the test phantom. (b) Cross-sectional images of the reconstruction. (c)-(d) 3D volume renderings of the reconstruction. }
	\label{fig: LSI}
\end{figure}

\section{Harmonic Portraits and PSFs for MH3D} \label{sec: main}

\subsection{Harmonic Portrait Domain} 
For each $z$-slab, the 2D rastering pattern is identical, defined by the FFP trajectory $(x_\xi(t), y_\xi(t))$, giving a time-to-position mapping $t \leftrightarrow (x_\xi(t), y_\xi(t))$. We define a portrait image by gridding the time signal back to this spatial domain:
\begin{equation}\label{eq: slab_def}
	d(x_\xi (t) , y_\xi (t) , z_j ) = s(t ; \, z_j). 
\end{equation}
This requires downsampling $s(t; z_j)$ once per transmit period.

To construct harmonic portraits, we isolate each harmonic using digital filters (see Appendix, equation (\ref{eq: skfilt2})), yielding $s_k$, which is the received signal filtered for the $k$th harmonic band. Then the harmonic portraits are given by
\begin{equation}\label{eq: port}
	d_k (x_\xi (t) , y_\xi (t) , z_j) = s_k(t ; \, z_j).
\end{equation}
{\color{black}Examples of harmonic portraits are shown throughout the manuscript in Figures \ref{fig: portraits}, \ref{fig: PSFs}, and \ref{fig: 3Dspiral}}. Note that this gridding transformation is effectively lossless, as most signal energy lies in the harmonics, and thus the full time-domain signal can be accurately reconstructed from the portraits. This enables full image recovery from harmonic portraits.

\subsection{Portrait PSFs and MH3D Model Overview} 
The complete mathematical model for our harmonic portrait image data is formally derived in the appendix, and the main ideas are outlined here. Each harmonic portrait is modeled as a vector-valued convolution given by\footnote{The vector valued convolution-dot product used here is defined as $\vec f * \vec g (x) = \sum_{i=1}^N f_i * g_i (x)$.}
\begin{equation}\label{eq: MH3Dcont}
	d_k(x,y,z_j ) = \vec h_k * (\rho \cdot \vec b_1)(x,y,z_j), 
\end{equation}
where $\vec h_k = (h_{kx}, h_{ky}, h_{kz})^\top$ are PSFs and $\vec b_1$ is the receive coil sensitivity. The harmonic PSFs are the portraits that would be obtained by scanning a point source at the origin. For example, if $\vec b_1 = (1,0,0)$ and the SPION object is a point source, then we would effectively measure $h_{kx}$ in the $k$th harmonic portrait. One may then think of a more general harmonic portrait with non-uniform receive field $\vec b_1$ as containing a combination of point sources and sensitivity weights giving rise to the convolutional model above.  

The discretized matrix model across harmonics $k = 2, \ldots, K$ then becomes
\begin{equation} \label{eq: matrixModel}
	\begin{bmatrix}
		d_2 \\ \vdots \\ d_K
	\end{bmatrix}
	=
	\begin{bmatrix}
		H_{2x} & H_{2y} & H_{2z}\\
		\vdots & \ddots & \vdots \\
		H_{Kx} & H_{Ky} & H_{Kz}
	\end{bmatrix}
	\begin{bmatrix}
		B_x \\ B_y \\ B_z
	\end{bmatrix}
	\rho,
\end{equation}
where 
\begin{itemize}
	\item $H_{kx}$, $H_{ky}$, and $H_{kz}$ are circulant matrices that represent the convolutional operations for each of the 3 components of $\vec h_k$.
	\item $B_x$, $B_y$, and $B_z$ are diagonal matrices containing the receive coil sensitivity weights in $\vec b_1$.  
\end{itemize} 
All operators are implemented using memory-efficient, matrix-free methods \cite{sanders2020effective, sanders2025physics}.

\begin{figure}
	\centering
	\includegraphics[width=0.8\textwidth]{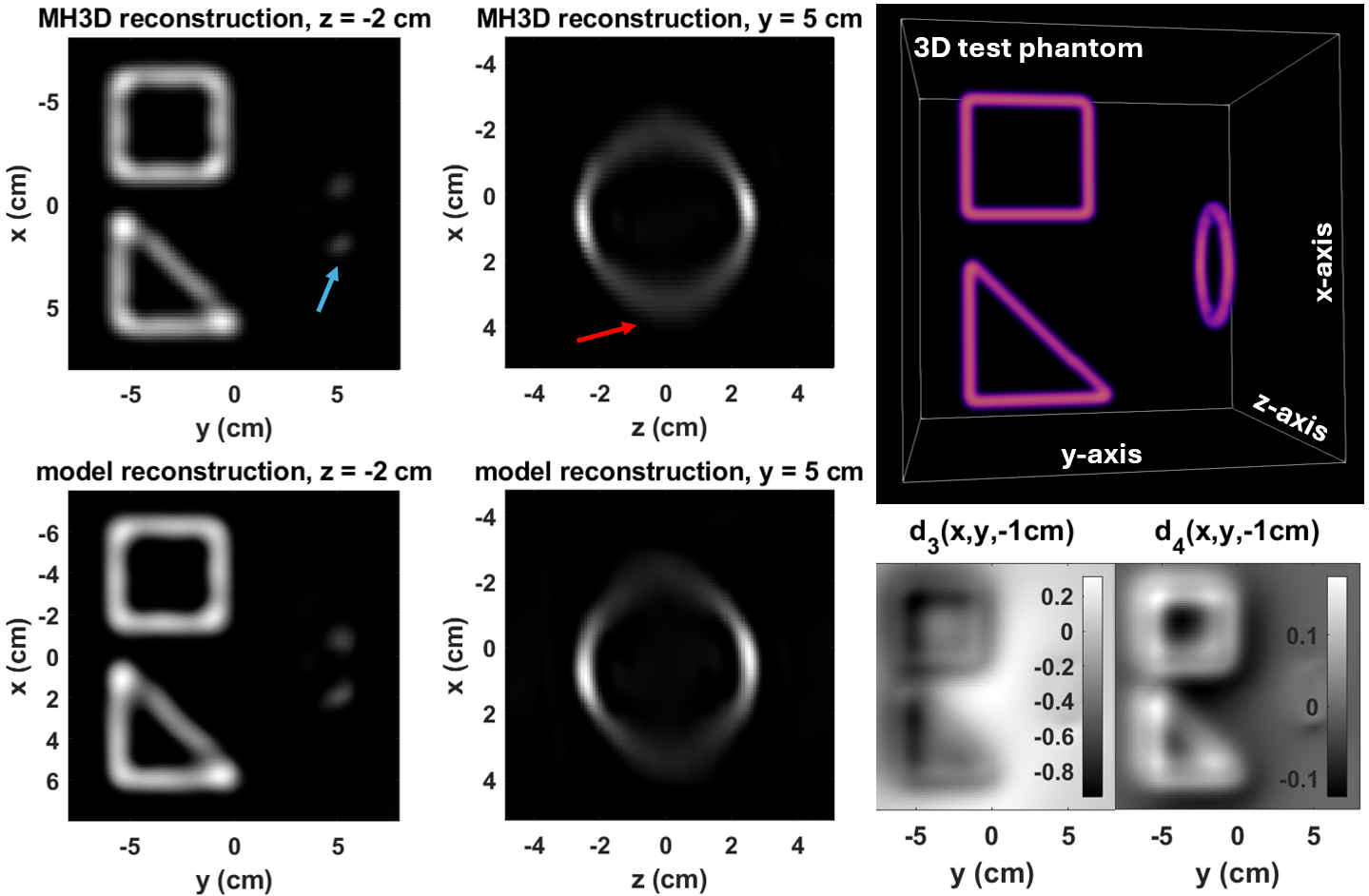}
	\caption{Comparison between model-based reconstruction \cite{sanders2025physics} and MH3D for a simulated phantom. Minor differences arise from solver details and parameter tuning. The red arrow indicates some mild anisotropic blurring in the reconstruction inherently due to a $z$-axis only transmit. The blue arrow indicates where the circle passes through the $z=-2$ cm plane and appears as two dots. The bottom right images show the 3rd and 4th harmonic portraits at $z=-1$ cm. The top right image shows a 3D visualization of the simulated phantom.}
	\label{fig: model}
\end{figure}

Figure \ref{fig: model} shows MH3D and model-based reconstructions of a phantom with square, triangle, and circle objects in orthogonal planes. The receive coil sensitivity profile given by $\vec b_1(\vec x) = (1,1,1)$ for all $\vec x$. While the full algorithmic details are deferred to later sections, this example is provided as a proof of concept and to motivate the discussion that follows. Despite minor differences, both methods yield comparable image quality, suggesting the MH3D model effectively captures all usable signal information.

In some of our applications, only the $z$-component of sensitivity is significant. In that case, the model simplifies to
\begin{equation}\label{eq: portz}
	d_k(x,y,z_j) = (h_k * (b_z \cdot \rho))(x,y,z_j),
\end{equation}
and the matrix form becomes
\begin{equation}\label{eq: matrixModelz}
	\begin{bmatrix}
		d_2 \\ \vdots \\ d_K
	\end{bmatrix}
	=
	\begin{bmatrix}
		H_{2z} \\ \vdots \\ H_{Kz}
	\end{bmatrix}
	B_z \rho.
\end{equation}

\subsection{Harmonic PSF Approximations}\label{sec: PSF}
\begin{figure}
	\centering
	\includegraphics[width=0.8\textwidth]{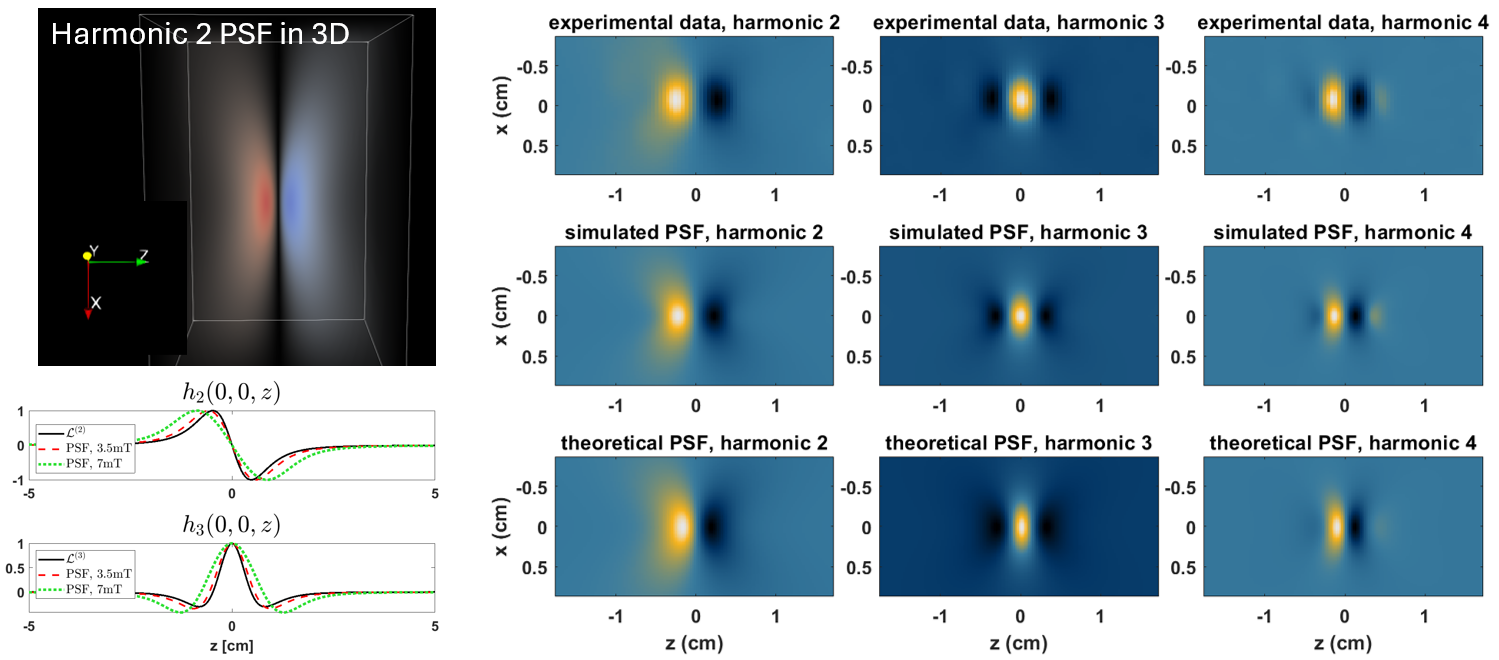}
	\caption{{\color{black}Illustration of harmonic PSF shapes. The right panels show harmonic portrait images from a 2D FFL Momentum  (Magnetic Insight, Alameda, CA) scan with a z-axis drive and receive chain at 20 mT drive amplitude. The top row contains experimental measurements of a point source, while the middle and bottom rows compare these data with simulated and theoretical PSFs, showing good agreement across harmonics 2–4. The top left panel shows a 3D rendering of the second-harmonic PSF from our 3D FFP scanner, where red and blue denote positive and negative values, respectively. The bottom left panels plot 1D PSF cross-sections at $x=y=0$ for different drive amplitudes, alongside approximations based on derivatives of the Langevin function. All curves are scaled by their respective maxima.}}
	\label{fig: PSFs}
\end{figure}

{\color{black}To most accurately model the PSFs, $\vec h_k$, used in the MH3D model, in practice we simulate the full scanner system model parameters with a uniform sensitivity. Gridding the harmonics from these simulated data onto the imaging mesh then obtains an accurate PSF used in the model.} While they can be used as black-box kernels, insight into their analytical structure has proven useful in the development of this work, including additional modeling parameters, number of harmonics needed, validity of the methodology, and development of alternative methods (see sections \ref{sec: alpha} and Appendix \ref{sec: MHAD}). 

We provide the general concept behind the harmonic PSFs in a 1D setting, while the full details are provided in the appendix. Consider the model in equation (\ref{eq: portz}) with a uniform sensitivity now written as
\begin{equation}\label{eq: hkapprox0}
	d_k(x) = h_k * \rho (x). 
\end{equation}
Then our claim (see appendix) is that the 1D harmonic PSF $h_k$ is approximated by
\begin{equation}\label{eq: hkapprox}
	h_k(x) \approx c_k \mathcal L^{(k)} (\beta x ) ,
\end{equation}
where $\mathcal L$ is the Langevin function and $c_k$ are unique coefficient constants. In 3D with $z$-axis drive/receive, the harmonic PSFs generalize to $\frac{d^k}{dz^k} h_{33}(x,y,z)$, where $h_{33}$ is the last element of the MPI PSF tensor \cite{goodwill2011multidimensional}. This approximation holds best at low transmit amplitudes (e.g. $<10$ mT, see Figure \ref{fig: PSFs}). {\color{black}This can be understood in part from the condition $|\gamma A|< \pi$ from Theorem 1 in the appendix.}

{\color{black}
Figure \ref{fig: PSFs} illustrates the PSFs obtained experimentally, in simulation, and from theoretical approximations based on derivatives of the Langevin function. The experimental data (top row) show harmonic PSFs of a point source from a 2D FFL Momentum scanner (Magnetic Insight, Alameda, CA) at 20 mT, while the simulated and theoretical results (middle and bottom rows) display close agreement with these measurements across harmonics 2–4. The bottom-left panels show 1D cross-sections of the 3D PSFs at $x=y=0$, highlighting that the Langevin derivative approximation is most accurate at lower drive amplitudes. Quantitatively, we typically observe relative fitting errors between the data and the reconstruction model of 10–20\% for high SNR data sets, which is consistent with the level of agreement visible in Figure \ref{fig: PSFs}.
	
}

\section{Additional Modeling Considerations}\label{sec: additional}
{\color{black}
This section outlines the practical components of the MH3D reconstruction framework, each addressing a specific aspect of data handling and signal processing on our scanner. In Section \ref{sec: sampling}, we describe downsampling strategies that preserve model accuracy while maintaining adequate resolution. Section \ref{sec: bdy} describes a related boundary padding method used to mitigate edge artifacts, and Section \ref{sec: alpha} presents a practical regularization scheme to reduce artifacts arising from the missing first harmonic. Section \ref{sec: runtime} outlines the numerical methods and demonstrates that the algorithm reconstructs 3D volumes in only seconds. Finally, in Section \ref{sec: phase}, we detail the phase calibration procedure that ensures model consistency across harmonics.
}

\subsection{Downsampling Methods for Accurate Modeling}\label{sec: sampling}
In an ideal setting, the portrait data $d_k(x,y,z)$ would be sufficiently sampled in the discretized $(x,y,z)$ space to approximate the true portrait in the continuous space. However, in our system and many others, the $z$-axis is encoded via stepped mechanical translation, and due to the desire to limit scan acquisition times, the number of scanned $z$-locations is typically kept relatively low. In our case, the default distance between neighboring $z$-locations used is $\Delta z =5$ mm. At this distance the discretization of the continuous convolutional model along the $z$-axis can result in an undersampled model. The right images of Figure \ref{fig: mesh} demonstrate this by showing the PSF approximations when sampled at 1 mm and then at 5 mm. It can clearly be seen here that using the 5 mm mesh would notably undersample the PSF.

\begin{figure}
	\centering
	\includegraphics[width=0.75\textwidth]{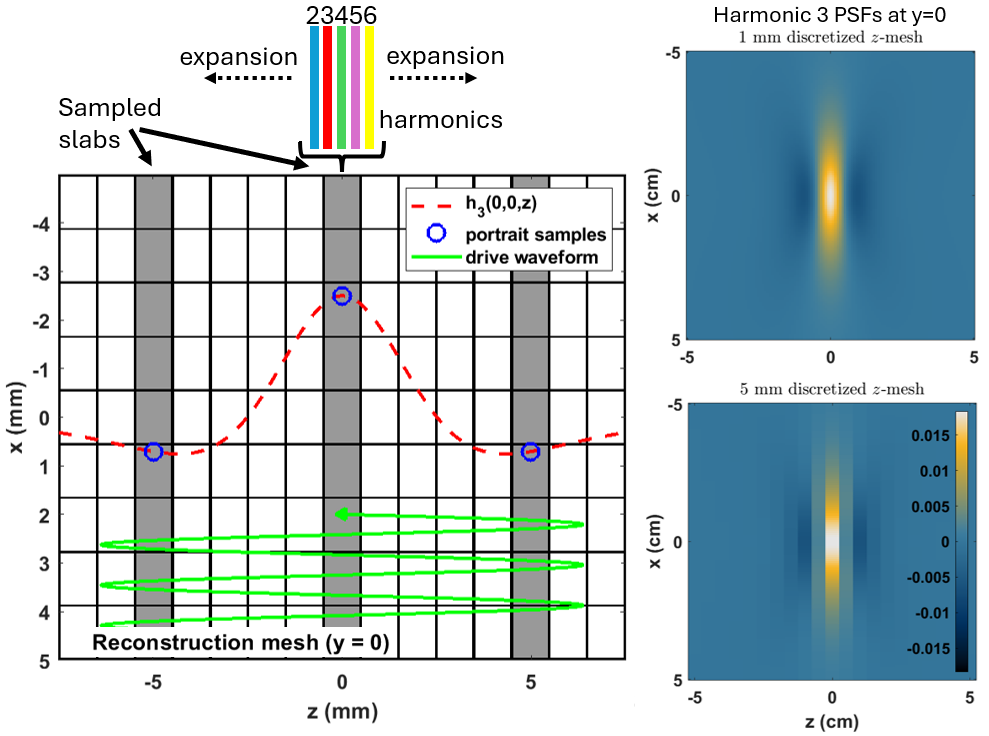}
	\caption{Imaging reconstruction mesh (left) showing the sampled locations shaded in gray. The forward model uses the full high resolution mesh, while only downsampling to the portrait slabs as the final operation. The 3rd harmonic PSF at $x=y=0$ is overlaid onto the mesh to demonstrate how discretizing the full forward model at only the scanned locations would undersample the PSFs, resulting in an inaccurate model. This is also demonstrated in the right images showing the PSFs at 1 mm resolution and 5 mm resolution in $z$.  However, with multiple harmonics, the model can expand the information outward to positions between the measured locations.}
	\label{fig: mesh}
\end{figure}

To overcome this issue, also demonstrated in Figure \ref{fig: mesh}, we retain a high resolution model for all of the forward operations, while simply applying a downsampling operation to the measured $z$ locations as the final component of our model. In this way we are able to accurately approximate the continuous model at high resolution, while only downsampling to the actual scanned locations at the final stage. We have empirically found $\Delta z \approx 1$ mm to be sufficient to accurately approximate the continuous model. Extending the matrix model given in (\ref{eq: matrixModel}) to this downsampling method arrives at
\begin{equation}\label{eq: dN}
	\underbrace{
		\begin{bmatrix}
			d_2 \\ d_3\\\vdots \\ d_K
		\end{bmatrix}
	}_{:=\mathbf d}
	= \underbrace{\underbrace{\begin{bmatrix}
				P & 0 & \dots & 0\\
				0 &  P & \dots & 0\\
				\vdots & & \ddots & 0\\
				0 & 0 & \dots & P
			\end{bmatrix}
		}_{:= \mathbf P}
		\underbrace{
			\begin{bmatrix}
				H_{2x} & H_{2y} & H_{2z}\\
				H_{3x} & H_{3y} & H_{3z}\\
				\vdots & \ddots & \vdots \\
				H_{Kx} & H_{Ky} & H_{Kz}\\
			\end{bmatrix}
		}_{:= \mathbf H}
		\underbrace{
			\begin{bmatrix}
				B_x \\ B_y \\B_z
		\end{bmatrix}}_{\mathbf B}
	}_{\mathbf A:= \mathbf{P H B}}
	\rho ,
\end{equation}
where $P$ downsamples the result output from the first two high-resolution operators to only those $z$ locations acquired in the scan. The full operator, $\mathbf P$, is written as a banded diagonal matrix, since this operation is applied for each harmonic. In matrix form, $P$ is called a \textit{row-selector} matrix, which contains only some rows of the identity matrix, i.e. a downsampling operation.

Note that using this model for the reconstruction also recovers $\rho$ at the higher-resolution as defined by the first two operations, which generally improves the visual image quality. If we were using only a single harmonic, this added resolution may only provide additional pixels but not actual added resolution. However, the additional harmonics provides feasibility that this resolution is actually attainable, which we have observed empirically in simulations and with experimental data. The actual attainable resolution is highly system and scan dependent and should therefore be assessed accordingly.

One can also interpret this reconstruction process as a type of regularized expansion along the $z$-axis, implicitly using the information encoded in the multiple harmonics. According to equation (\ref{eq: hkapprox}), each harmonic portrait can be viewed as encoding successive derivatives of the underlying signal. This allows the reconstruction to approximate unsampled locations by leveraging a Taylor-like expansion. For example, from equations (\ref{eq: hkapprox0}) and (\ref{eq: hkapprox}), we can estimate the signal at a neighboring point $x + \delta$ as:
\begin{equation}\label{eq: tlr}
	\mathcal L_\beta' * \rho(x + \delta) \approx \sum_{k=1}^K c_k^{-1} d_k(x) \frac{\delta^k }{ k! },
\end{equation}
where $\mathcal L_\beta'(x) = \mathcal L'(\beta x)$ and $d_k$ is the $k$-th harmonic portrait. Thus, the MH3D model enables extrapolation between sampled slabs using higher-order harmonic information\footnote{Note that the $k=1$ term is missing due to system filtering, as discussed in Section \ref{sec: alpha}.}.

Alternatively, from a linear algebra perspective, given $K-1$ harmonics, we can, in theory, resolve up to $K-1$ independent locations per scanning position. More precisely, suppose we collect $N_z$ portrait slabs. Then using harmonics 2 through $K$ provides $(K-1)N_z$ effective measurements. A square linear algebra system is obtained by setting the reconstructed image resolution to $N_x \times N_y \times (K-1)N_z$. For instance, using 5 harmonics ($K=6$) with slabs sampled at $\Delta z = 5$ mm yields an effective 1 mm resolution in $z$.

Finally, the underlying intuition aligns with classical sampling theory. As Shannon noted in his foundational work \cite{shannon1949communication}, "One can further show that the value of the function and its derivative at every other sample point are sufficient. The value and first and second derivatives at every third sample point give a still different set of parameters which uniquely determine the function. Generally speaking, any set of 2TW independent numbers associated with the function can be used to describe it." MH3D leverages a similar principle, with each harmonic contributing derivative-like information that enables reconstruction beyond the raw sampling rate.

\subsubsection{Boundary Image Padding}\label{sec: bdy}
To reduce wrap-around artifacts from FFT-based convolutions, we apply boundary padding to the imaging FOV. Since FFTs are inherently periodic, unpadded convolutions can introduce edge distortions. To mitigate this, we extend the FOV with extra pixels along the boundaries, which are included in the forward model but cropped after reconstruction. These padding regions are not part of the portrait data and are discarded in post-processing. The extended model follows the same formulation as equation (\ref{eq: dN}), so we omit a separate derivation.

\begin{figure}
	\centering
	\includegraphics[width=0.8\textwidth]{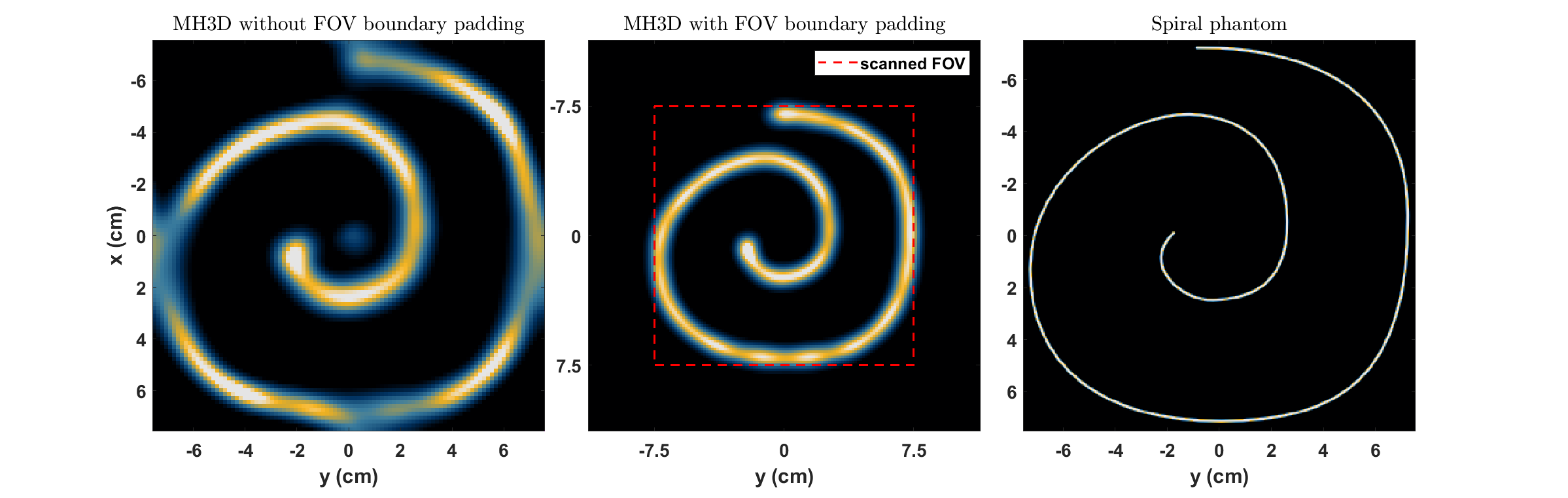}
	\caption{MH3D reconstructions with and without padding for a 2D spiral phantom near the FOV edge. Padding eliminates distortion and edge artifacts. Pixels outside the red FOV box are discarded.}
	\label{fig: bdPad}
\end{figure}

Figure \ref{fig: bdPad} shows reconstructions with and without padding. In the unpadded case, SPIONs near the edge cause distortion, blur, and spurious signals at the center. Padding removes these effects and improves image quality. Outer pixels, shown outside the red box, are discarded after reconstruction.

\subsection{Image Baselining and Artifact Reduction}\label{sec: alpha}
Given data vector $\mathbf d$ and forward model $\mathbf A$ as in (\ref{eq: dN}), we solve the inverse problem
\begin{equation}\label{eq: Amodel}
	\rho^* = \arg \min_{\rho \ge 0}\| \mathbf A \rho - \mathbf d \|_2^2 + \lambda R(\rho),
\end{equation}
where $R$ is a regularizer based on prior assumptions about $\rho$. In other applications we have often used a second-order Tikhonov term,
\begin{equation}
	R(\rho) = \| T \rho \|_2^2,
\end{equation}
with $T$ a finite-difference operator to encourage smoothness \cite{sanders2020effective}. 

\begin{figure}
	\centering
	\includegraphics[width=0.8\textwidth]{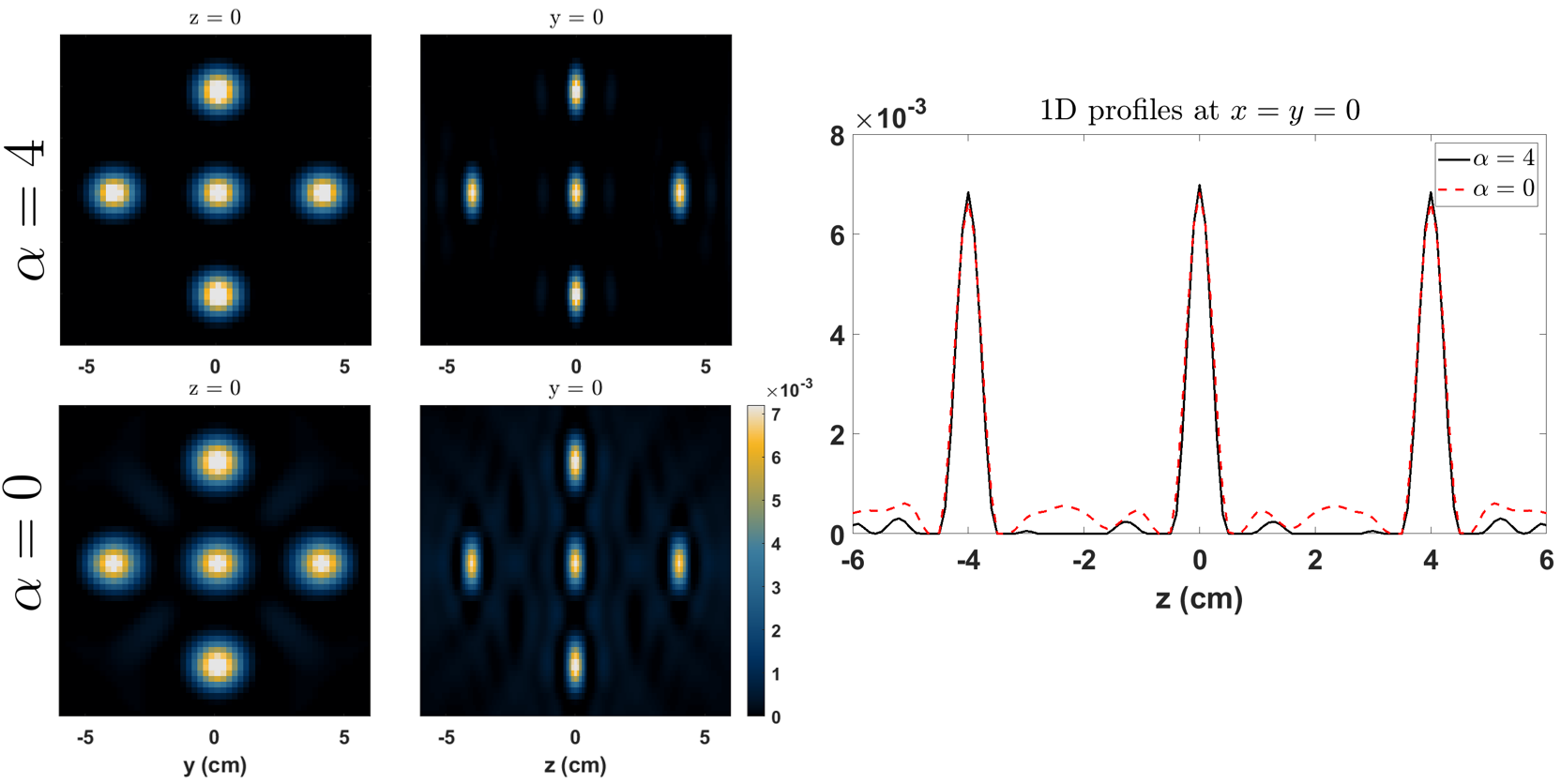}
	\caption{Simulated example demonstrating the effect of the $\alpha$ parameter on background suppression and artifact reduction. A 3D phantom of point sources is shown with $\alpha = 0$ (left) and $\alpha = 4$ (right). The phantom image contains point source objects of uniform concentration, with one source $(0,0,0)$ and the others located on the sphere of radius 4 cm centered around the origin.}
	\label{fig: alpha}
\end{figure}

However, MPI harmonic data lacks a DC component due to filtering, resulting in an unknown background level. In MH3D, this emerges from the structure of the PSFs. For example, in 1D with a uniform sensitivity $b_1 = 1$, the second harmonic portrait is approximately the second derivative of the Langevin PSF convolved with the nanoparticle density:
\begin{equation}
	d_2(x) = \mathcal L'' * \rho(x) = \frac{d}{dx} \left[\mathcal L' * \rho(x) \right].
\end{equation}
Recovering the native image $\rho_N(x) = \mathcal L' * \rho(x)$ via integration yields
\begin{equation}\label{eq: 2AD}
	\rho_N(x) = \int d_2(x) \, dx = \mathcal L'*\rho(x) + C,
\end{equation}
where $C$ is an unknown constant. This constant is analogous to the missing background and edge zero-pinning implemented in classical X-space MPI \cite{goodwill2011multidimensional}.

To compensate, we introduce a prior that penalizes intensity near the image boundaries. Let $P_\alpha$ be a row-selector matrix that selects these boundary pixels. Then we define the regularizer as
\begin{equation}\label{eq: finalR}
	R(\rho) = \| T \rho \|_2^2 + \alpha \| P_\alpha \rho \|_2^2.
\end{equation}
This $\alpha$ term enforces prior knowledge of the background constant and reduces ringing artifacts. Figure \ref{fig: alpha} shows a simulated example where increasing $\alpha$ drives the background closer to zero, improving overall image quality.

\subsection{Computational Methods and Runtime Performance}\label{sec: runtime}
{\color{black}
\begin{table}[t]
	\centering
	\caption{Efficient implementations of operators in the MH3D solver.}
	\label{tab:ops}
	\renewcommand{\arraystretch}{1.2}
	\begin{tabular}{lp{9cm}}
		\hline
		\textbf{Operator} & \textbf{Efficient implementation / Notes} \\
		\hline
		$\mathbf{B}$ & Diagonal matrix (Hadamard product) with precomputed coil map(s). \\
		$\mathbf{H}_k$ & FFT-based convolution: $u \mapsto \mathcal{F}^{-1}(\widehat H_k \odot \mathcal{F}u)$ with $\widehat H_k$ precomputed. \\
		${P} ,{P}_\alpha$ & Row selection / gather (no arithmetic). \\
		${T}^\ast{T}$ & Fourier diagonalization: ${T}^\ast{T}=\mathcal{F}^{-1}|\widehat T|^2\mathcal{F}$ (two FFTs). \\
		\hline
	\end{tabular}
\end{table}
At this stage, the full MH3D framework has been established, and we turn to the computational methods required to solve the reconstruction problem. Considering equations (\ref{eq: Amodel}) and (\ref{eq: finalR}), the optimization problem takes the form
\[
\min_{\rho \ge 0}\| \mathbf A \rho - \mathbf d \|_2^2 + \lambda \cdot \left(\| T \rho \|_2^2 + \alpha \| P_\alpha \rho \|_2^2\right).
\]
This can be equivalently expressed as a single least-squares problem:
\[
\min_{\rho \ge 0} \| M \rho - b \|_2^2,
\]
with
\[
M = 
\begin{bmatrix}
	\mathbf{A} \\ \sqrt{\lambda}\, T \\ \sqrt{\lambda \alpha}\, P_\alpha 
\end{bmatrix},
\quad
b = 
\begin{bmatrix}
	\mathbf d \\ \vec 0 \\ \vec 0
\end{bmatrix}.
\]

The optimization is solved via accelerated projected gradient descent (e.g., \cite{beck2009fast,su2016differential}). The $(k+1)$-st iteration is given by
\begin{equation}
	\begin{split}
		y^{k} &= \rho^k + \frac{k-1}{k+2}(\rho^k - \rho^{k-1}), \\
		\rho^{k+1} &= \max( y^k - \tau M^*(M y^k - b), \; 0).
	\end{split}
\end{equation}

The overall efficiency is dominated by the cost of evaluating matrix-vector products with $M$ and $M^*$. To that end, each of the operators involved can be computed efficiently by either elementwise Hadamard products or FFT-based convolutions (see the methods in \cite{sanders2018multiscale,sanders2020effective} for details), and they are further significantly accelerated by evaluation on modern GPU hardware. A brief summary of how each operator is evaluated is provided in Table \ref{tab:ops}.

\begin{table}[t]
	\centering
	\caption{Runtime evaluation of MH3D reconstructions under different test cases. The evaluations were run in MATLAB using an NVIDIA GeForce RTX 3070 Laptop GPU.}
	\label{tab:runtime}
	\renewcommand{\arraystretch}{1.2}
	\begin{tabular}{ccccc}
		\hline
		\textbf{\# z-slabs} & \textbf{Image dimensions} & \textbf{\# Harmonics} & \textbf{Iterations} & \textbf{Total runtime} \\
		\hline
		21  & $100 \times 100 \times 121$   & 4 & 500 & 27.0 sec \\
		31  & $100 \times 100\times 151$   & 4 & 500 & 35.0 sec \\
		41 & $100 \times 100 \times 241$  & 4 & 500 & 46.4 sec \\
		\hline
	\end{tabular}
\end{table}

Representative runtimes for three test cases are summarized in Table~\ref{tab:runtime}. These results demonstrate that MH3D reconstructions can be performed very efficiently on standard workstation hardware, reconstructing high 3D spatial resolutions in only seconds.
}

\subsection{Portrait Phase Calibration}\label{sec: phase}

Harmonic filtering returns complex-valued signals $s_k$, resulting in complex harmonic portraits $d_k$. However, our model assumes real-valued PSFs and portraits (e.g., Figure~\ref{fig: PSFs}), so phase correction is necessary.

Several system factors introduce phase shifts including analog filters, harmonic filter design, transmit waveform phase, and receive chain delays. These factors introduce an unknown but measurable phase change in the portraits. Tracking individual phase contributors is impractical. Therefore, rather than modeling each source of phase error, we apply a simple correction directly to the portraits.

Empirically, we observe that the phase across each narrow harmonic band is approximately constant. Thus, we model the corrupted portrait with a constant phase offset by
\begin{equation}
	\widetilde{d}_k = e^{i\theta_k} d_k,
\end{equation}
where $\theta_k$ is the unknown constant phase.

We estimate $\theta_k$ by measuring the angle of $\widetilde{d}_k$ with high SNR samples, i.e.
\begin{equation}
	\theta_k \approx \text{angle}(\widetilde{d}_k(x,y,z)),
\end{equation}
using hand-selected regions. Because the true portrait may have both positive and negative values, a phase ambiguity of $\pi$ may remain. We resolve this by testing a reconstruction with one phase value; if necessary, we flip the phase. This calibration is performed once per receive channel and remains stable due to closed-loop hardware synchronization.


\section{Results}
\subsection{Experimental Data}
We provide results using our MH3D reconstruction model on experimental data containing VivoTrax (Magnetic Insight, Inc., Alameda, CA.) with the 3D FFP scanner described in Section \ref{sec: scanner}. The $z$-slab spacing for all data sets is $\Delta z = 5$ mm, and the reconstructions utilized harmonic portraits 2-5.

{\color{black}
	The first example, shown in Figure~\ref{fig: 3Dspiral}, is a test phantom consisting of a 3D helical spiral taped around a head mannequin, covering most of the imaging FOV. Maximum intensity projections (MIPs) are displayed in panels (a) and (b), where the maximum is taken along a single axis, and a photograph of the phantom is provided in panel (c) for comparison. In this experiment, we compare MH3D with a model-based reconstruction and with a reconstruction using only the 3rd harmonic, a common approach in prior work \cite{mason2022side,janssen2022single,nomura2024development,mcDonough2024tomographic}. The model-based reconstruction could only be performed by first applying artifact correction and phase calibration in the portrait domain, followed by mapping the portraits back to the time-domain. A side effect of this procedure is a small signal dropout in this example, highlighted by the red arrow in panel (a). Otherwise, the model-based reconstruction is qualitatively similar to MH3D, with slightly sharper resolution. In contrast, the 3rd-harmonic-only reconstruction, while implemented within our MH3D framework, shows substantial blurring and artifacts, underscoring the importance of incorporating multiple harmonics. This example also illustrates the potential of MPI for future head imaging applications, facilitated by advanced reconstruction tools such as MH3D.
}

\begin{figure}
	\includegraphics[width=1\textwidth]{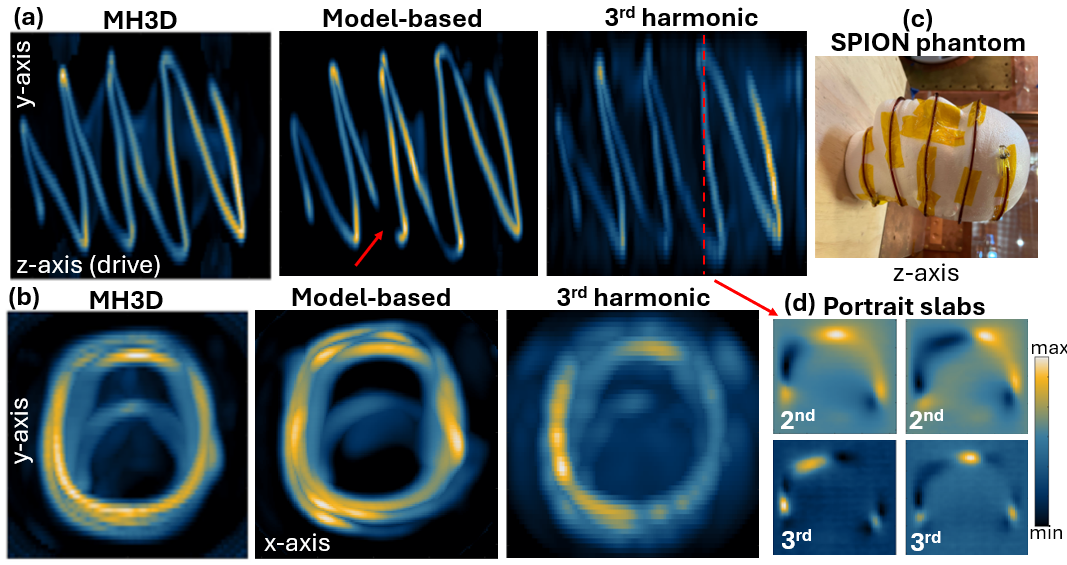}
	\caption{{\color{black}Reconstructions of a 3D helical spiral phantom taped to a human head mannequin, comparing MH3D with a model-based reconstruction and with a reconstruction using only the 3rd harmonic. (a–b) Maximum intensity projections (MIPs) of the reconstructions along the x-axis (a) and z-axis (b). (c) Photograph of the phantom. (d) Portrait slabs for the 2nd and 3rd harmonics extracted along the red dashed line in (a). The model-based reconstruction required artifact correction and phase calibration in the portrait domain before mapping back to the time-domain, which resulted in a signal dropout (red arrow).}}
	\label{fig: 3Dspiral}
\end{figure}

{\color{black}
Figure~\ref{fig: DR} presents a second experimental example designed to evaluate image quality under a high dynamic range with clinical applications in mind \cite{sehl2024magnetic}. A head phantom mimics lymph-node imaging after a tracer injection: the injection site is placed at the top of the head (28 mg VivoTrax), three mock lymph nodes are located near the left ear (each 100 $\mu$g), and a reference vial (700 $\mu$g) is mounted on the right side; the phantom is shown in the photograph in the lower-right panel. The MIP reconstruction images comparing MH3D and the model are shown with adjusted contrast to visualize the wide dynamic range. These images show that both MH3D and the model-based method clearly resolve all targets. A notable difference is a localized distortion at the strong injection site in the model-based result, which is again a likely artifact introduced when portrait-domain corrections are mapped back to the time-domain for the model reconstruction.
}

\begin{figure}
	\centering
	\includegraphics[width=0.75\textwidth]{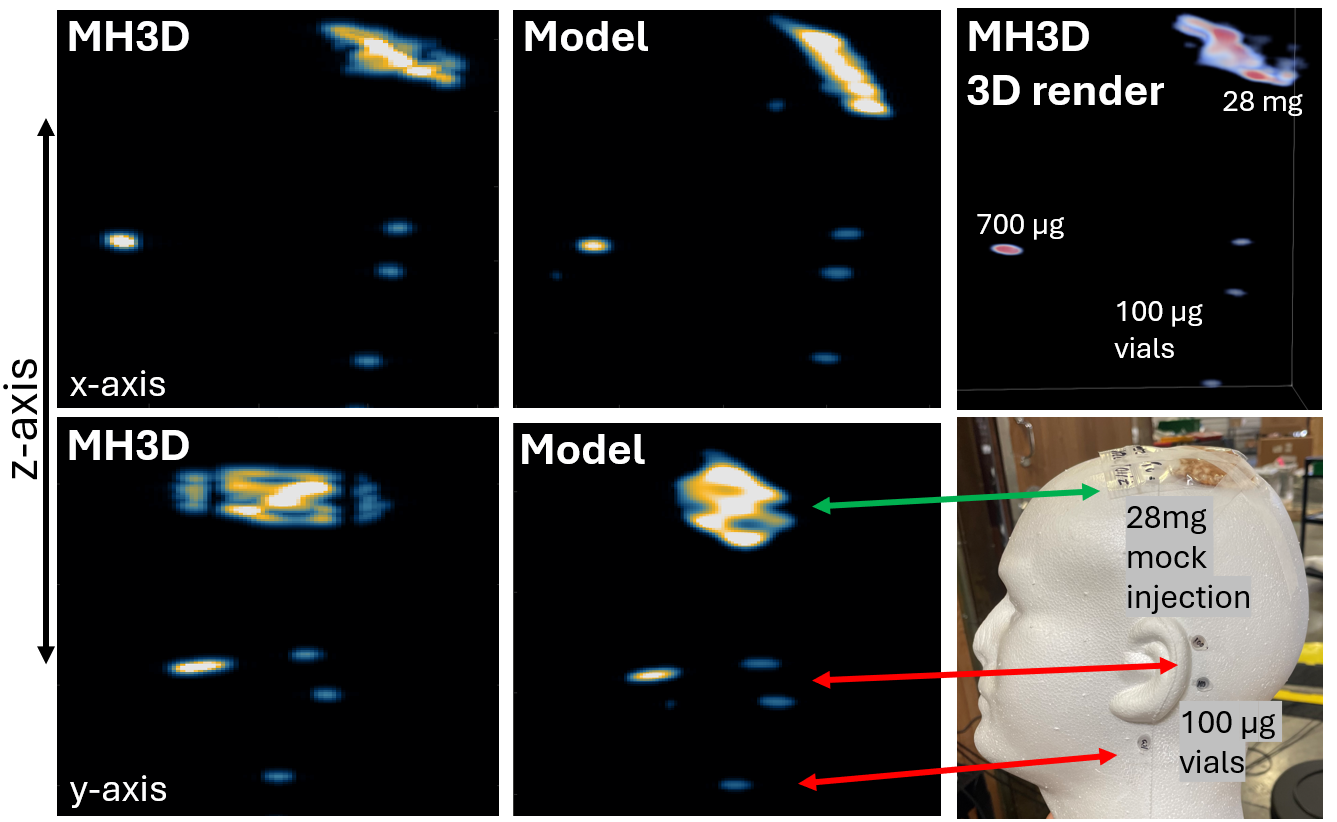}
	\caption{Experimental example on a mock lymph node imaging phantom with an injection comparing MH3D and the model reconstruction. The photo on the bottom right shows the phantom and the arrows point to the corresponding objects in the phantom found in the reconstructed images.}
	\label{fig: DR}
\end{figure}

\subsection{Simulated 2D Projection Reconstruction}
In this section, we present a simulated example of MH3D applied to a 2D FFL projection geometry, which we refer to as MH2D. This demonstration highlights how the MH3D framework can be generalized to alternative scanning geometries and enables quantitative comparison with other reconstruction methods. The simulations mimic the acquisition scheme of a production preclinical MPI scanner (Momentum, Magnetic Insight, Alameda, CA). In this scanner, the FFL focus fields raster in a zigzag Cartesian-like trajectory. The focus field shifts in the $xz$ plane while the FFL direction is projected along the $y$-axis. This scan pattern is repeated twice, with the drive field alternating between the $x$- and $z$-axes for each pass, and the receive coil aligned collinearly with the driven axis.

{\color{black}
The test image used for this simulation is shown on the right of Figure~\ref{fig: mh2d} and consists of a spiral arrangement of small point-like sources. To compare reconstruction methods, we applied each algorithm over a wide range of regularization parameters ($\lambda$), with significant white noise added to the simulated signal data to ensure all reconstructions exhibited unavoidable noise artifacts. Two representative cases, $\lambda = 2.6$ and $\lambda = 10^3$, are shown in the top panels of Figure~\ref{fig: mh2d}. As in previous examples, we compare MH2D with the model-based reconstruction \cite{sanders2025physics} and with a 3rd-harmonic-only reconstruction.

The right panel of Figure~\ref{fig: mh2d} shows three quantitative metrics across the tested $\lambda$ values. From top to bottom, these are: the full-width-half-maximum (FWHM), a signal-to-noise ratio (SNR) defined as the average peak signal intensity across SPION sources divided by the background standard deviation, and an alternative SNR defined as the peak signal divided by the peak background noise. The first SNR metric is conventional, while the second highlights the likelihood of falsely interpreting background noise as signal. Across both images and metrics, the 3rd-harmonic-only reconstructions perform substantially worse than the alternative methods. The MH2D and model-based reconstructions appear visually similar at both $\lambda$ values, with nearly identical FWHM curves across the entire range. However, both SNR metrics show a modest advantage for the model-based approach, consistent with our own visual inspection, which revealed slightly more noise artifacts in the MH2D images after contrast adjustment. The cause of this difference is not fully understood, but we suspect it may reflect tuning refinements that have been extensively optimized for MH3D and the 2D model-based reconstructions but not yet for MH2D. Overall, in terms of image quality, object similarity, and FWHM, MH2D and the model-based method perform nearly equivalently.

For very small $\lambda$ (effectively no regularization), the metrics for all three methods converge, but the resulting images are not usable in practice because peak background noise approaches the peak signal. Hence, our analysis emphasizes the critical importance of regularization, which is standard practice in most imaging inverse problems.
}

\begin{figure}
	\centering
	\includegraphics[width=1\textwidth]{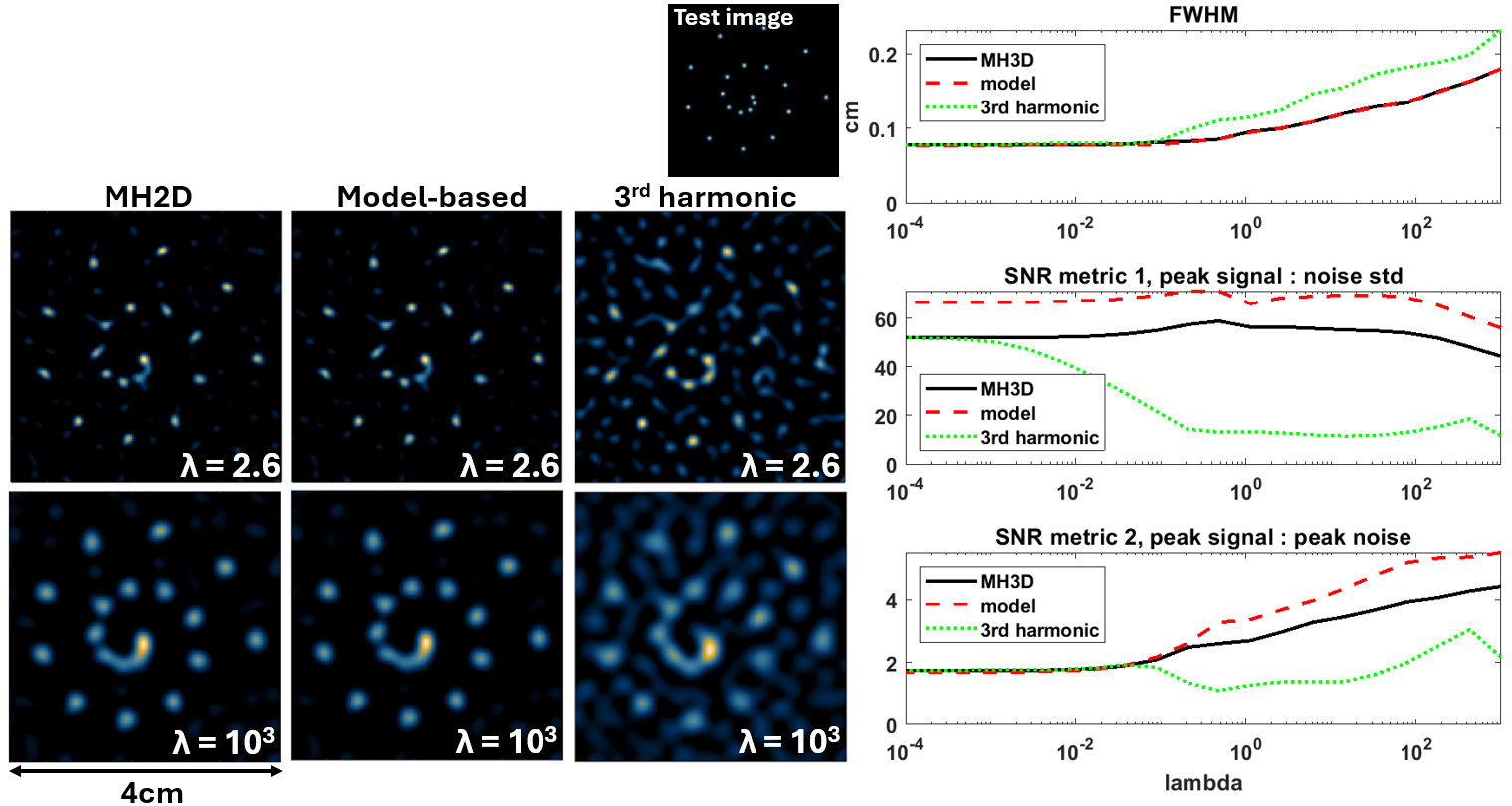}
	\caption{{\color{black}Simulated reconstruction study demonstrating the extension of the MH3D framework to a 2D projection scan geometry emulating data from a Momentum scanner (Magnetic Insight, Inc., Alameda, CA, USA). Reconstructions of a spiral test phantom are shown for MH2D, the model-based method, and a 3rd-harmonic-only approach at two regularization parameters ($\lambda = 2.6$ and $\lambda = 10^3$). The right panel summarizes quantitative metrics across a broad range of $\lambda$ values: (top-right) full-width-half-maximum (FWHM), (center-right) signal-to-noise ratio (SNR) defined by the mean peak signal over the background standard deviation, and (bottom-right) an alternative SNR defined by the peak signal over the peak background noise.}}
	\label{fig: mh2d}
\end{figure}

\section{Summary}
We have developed a new efficient model-based reconstruction for MPI to work in a common gridding domain, which we have coined as the \textit{portrait} domain. While this harmonic domain is not new to the field, we have developed a complete mathematical framework that is both computationally efficient and captures the full breadth of information in the MPI data. Our theoretical and empirical results show that the image reconstruction performance is comparable to a recently developed physics-based MPI model \cite{sanders2025physics}. However, the new method provided here offers significantly improved data visualization, facilitating enhanced analysis and debugging, which have proven invaluable in our work. The drawback of this new method is that it requires additional re-engineering for new scanner configurations and pulse sequences, while the former model is easily amenable. Therefore, we believe that both methods will continue to have powerful utility within the field of MPI. Finally, the development of the new MH3D model has led to invaluable new theoretical insights into the nature of the MPI signal and the information attained at each of the higher harmonics, which we lay out in detail in the appendix. 

\appendix

\section{Analysis of the MH3D Harmonic PSFs}\label{sec: analysis}
In this section we provide some theoretical analysis and insight into the harmonic PSFs arising from the MH3D model. These results allow us to understand and develop the image reconstruction models, while also providing new interesting theoretical insights into the harmonic analysis of MPI signals. 

In all of the results that follow, we consider a 1D MPI model\footnote{The results effectively apply in 2D and 3D, but would need to be modified so that the drive axis matches the axis of our 1D case used here. These derivations are provided later in the appendix.} with a static linear magnetic field given by $Gx$ and SPION(s) with a conglomerate coefficient $\beta$ (see Section \ref{sec: scanner} for details).  We also reduce the Langevin function coefficient constants to a single coefficient defined by $\myC := \beta \cdot G$, so that the SPION response to the applied magnetic field is
$$
\mathcal L [\beta   \| H \|] = \mathcal L[\myC (\xi(t) - x)].
$$
For simplicity in the derivations, we allow for our drive to take on a complex exponential waveform, which does not have an obvious real physical interpretation. However, it greatly simplifies the mathematical derivations, and similar but more tedious results are obtained with real trigonometric waveforms.

The following is the main result which provides an analytical decomposition of an MPI received signal into its harmonics. Among other things, this result shows us that the information at the $k$th harmonic effectively samples the $k$th derivative of the Langevin. The proof of this result is provided later in Appendix \ref{sec: proof} to keep the flow of these results. 
\begin{theorem}\label{prop main}
	Suppose the FFP trajectory is given by a complex trigonometric waveform and a linear shift focus field given by
	\begin{equation}
		\xi(t) = A e^{i2\pi f_0 t} + t \cdot \Delta_t,
	\end{equation}
	where $|\myC A| < \pi$ and $\Delta_t$ is a linear shift rate constant, and let the SPION density be as single point source written as $\rho(x) = \delta(x - x_0)$. Then the noise-free and unfiltered MPI received signal is given by
	\begin{equation}\label{eq: sExpansion}
		\begin{split}
			s_0(t)
			& = i2\pi m f_0\, b_1(x_0)\sum_{k=1}^\infty
			\frac{(\myC A)^{k}}{(k-1)!} \mathcal L^{(k)}\left[\myC (t \Delta_t - x_0)\right]{e^{i2\pi f_0 k t}} + O(\Delta_t ).
		\end{split}
	\end{equation}
\end{theorem}

To extend this result further into the harmonic domain and to derive the portrait model, we apply harmonic filtering using the result from Theorem \ref{prop main}. The harmonic filtered signals, $s_k(t)$, are defined by window filtering for the band of frequencies in $k$-space around the chosen harmonic of $s_0$. There are many possible choices for a window filter around the frequency band (see Figure \ref{fig: spectrum} for example), e.g. a Hann window or a top hat function. To complete the derivation we do not need to explicitly define this function, but we simply denote it by $g(t)$.
\begin{figure}
	\centering
	\includegraphics[width=0.6\textwidth]{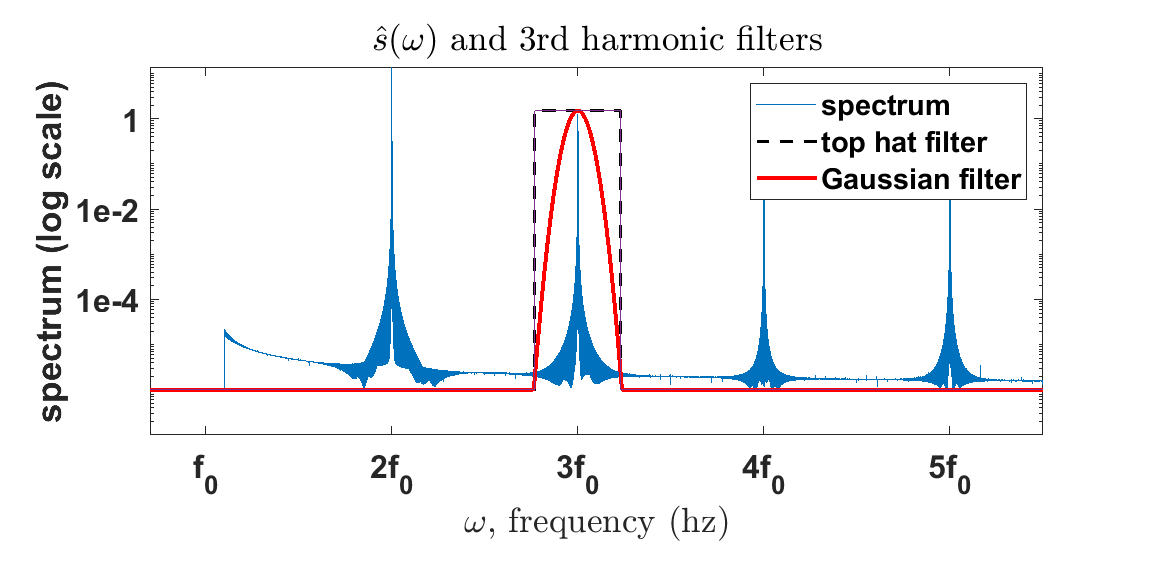}
	\caption{Example of a received signal's spectrum and possible window filters for the 3rd harmonic.}
	\label{fig: spectrum}
\end{figure}

\begin{definition}
	Given a received signal $s_0(t)$ and a window function $g(t)$, we define the $k$th harmonic filtered signal by
	\begin{equation}
		s_k(t) :=  \int_\R  \hat s_0 (\omega + k f_0) \hat g(\omega ) e^{i2\pi \omega t}   \, d \omega, 
	\end{equation}
	which may be reduced to
	\begin{equation}\label{eq: skfilt2}
		\begin{split}
			s_k(t) 
			& =  \int_\R  \hat s_0 (\omega + k f_0) \hat g(\omega ) e^{i2\pi \omega t}   \, d \omega \\
			& = \int_\R  \int_\R s_0 (\tau) e^{-i 2\pi (\omega + k f_0 ) \tau}\, d\tau \hat g(\omega) e^{i2\pi \omega t} \, d\omega \\
			& = \int_\R s_0 (\tau ) e^{-i2\pi k f_0 \tau} \left[
			\int_\R \hat g(\omega ) e^{i2\pi \omega (t - \tau) } \, d \omega
			\right] \, d\tau \\
			& =  \int_\R s_0 (\tau) e^{-i2\pi k f_0 \tau } g(t-\tau)  \, d \tau .
		\end{split}
	\end{equation}
\end{definition}

\begin{corollary}\label{mycor}
	Let $s_k(t)$ be defined as in (\ref{eq: skfilt2}), and assume the same conditions as in Theorem \ref{prop main}. Assume a homogeneous receive sensitivity ($b_1(x) = 1$) and $x_0 = 0$. Then the $k$th harmonic filtered signal for this point source is approximated by
	\begin{equation}
		s_{k,\delta}(t) \approx i 2 \pi m f_0 \frac{(\myC A)^k}{(k-1)!} 
		\int_\R \mathcal L^{(k)} (\myC  \tau \Delta_t) g(t- \tau ) \, d\tau.
	\end{equation}
	Therefore our harmonic PSF is given by gridding the harmonic time-domain signal back to FFP focus field location, obtaining
	\begin{equation}\label{eq: psf101}
		\begin{split}
			h_k(x) & = s_{k,\delta}(x\Delta_t^{-1} ) \\
			&   \approx i 2 \pi m f_0 \frac{(\myC A)^k}{\Delta_t (k-1)!} 
			\int_\R \mathcal L^{(k)} (\myC  \tau ) g((x - \tau) / \Delta_t) \, d\tau .
		\end{split}
	\end{equation}
\end{corollary}
\begin{proof}
	Combining (\ref{eq: skfilt2}) and (\ref{eq: sExpansion}) with a homogeneous receive coil obtains
	\begin{equation}\label{eq: app502}
		\begin{split}
			s_{k,\delta} (t) 
			& \approx  i2\pi m f_0 \sum_{j =1}^\infty \frac{(\myC A)^j}{(j -1)!} \int_\R \mathcal L^{(j)} (\myC \tau \Delta_t ) e^{i2\pi f_0 \tau (j -k)} g(t-\tau)  \, d\tau  \\
			& \approx i 2 \pi m f_0 \frac{(\myC A)^k}{(k-1)!} 
			\int_\R \mathcal L^{(k)} (\myC  \tau \Delta_t) g(t- \tau ) \, d\tau .
		\end{split}
	\end{equation}
	In these approximations, we have dropped the small $O(\Delta_t)$ term and used the fact that the integrals where $j\neq k$ are negligible due to the rapid decay properties of the Fourier transform for smooth functions\cite{iosevich2014decay}. 
	
	With the FFP focus point given by $t\Delta_t$, then the gridded harmonic portrait for this point source, which gives us the harmonic PSF, is given by
	\begin{equation}
		h_k(x) = s_{k,\delta}(x\Delta_t^{-1} ),
	\end{equation}
	which leads to (\ref{eq: psf101}). 
\end{proof}
We note that Corollary \ref{mycor} can be derived with equalities involving small additional terms related to the decay rates of the Fourier transform. However, the exact derivation significantly complicates the details without adding substantial value.

\subsection{Derivation of the Convolutional Portrait Model}
If $\rho(x)$ is a more general combination of point sources given as
\begin{equation}
	\rho(x) = \sum_n \rho_n \delta(x-x_n),
\end{equation}
then (\ref{eq: sExpansion}) generalizes to
\begin{equation}
	s_0(t) = 
	i2\pi m f_0 \sum_n b_1(x_n ) \rho_n \sum_{k=1}^\infty
	\frac{(\myC A)^{k}}{(k-1)!} \mathcal L^{(k)}\left[\myC (t \Delta_t - x_n)\right]{e^{i2\pi f_0 k t}} + O(\Delta_t ) .
\end{equation}
Using this expression and repeating similar steps to what was done in (\ref{eq: app502}) and substituting $\tau = \Delta_t^{-1} x_n + w$ obtains
\begin{equation}
	\begin{split}
		s_k(t) & \approx i2\pi m f_0 \frac{(\myC A)^k}{(k-1)!} \sum_n b_1(x_n ) \rho_n
		\int_\R \mathcal L^{(k)} (\myC  (\tau \Delta_t - x_n) )  g(t -\tau) \, d\tau \\
		& = i2\pi m f_0 \frac{(\myC A)^k}{(k-1)!} \sum_n b_1(x_n ) \rho_n
		\int_\R \mathcal L^{(k)} (\myC w \Delta_t ) g(t -\Delta_t^{-1} x_n - w) \, dw \\
		& = \sum_n b_1(x_n ) \rho_n 
		\underbrace{ i2\pi m f_0 \frac{(\myC A)^k}{(k-1)!} 
			\int_\R \mathcal L^{(k)} (\myC  w \Delta_t ) g(t -\Delta_t^{-1} x_n - w) \, dw }_{s_{k, \delta} (t - \Delta_t^{-1} x_n )} 
		\\
		& = \sum_n b_1(x_n )\rho_n s_{k, \delta} (t - \Delta_t^{-1} x_n).
	\end{split}
\end{equation}
Thus we have arrived at a relationship between the general harmonic filtered signal and the single point source harmonic filter signal resembling something similar to a convolution. To complete the portrait convolutional relationship, similar to above we again map $s_k$ to the portrait domain by gridding time points onto their FFP focus field locations, which obtains
\begin{equation}
	\begin{split}
		d_k(x) & = s_k(x \Delta_t^{-1} ) \\
		& = \sum_n b_1(x_n ) \rho_n s_{k,\delta}( (x - x_n) \Delta_t^{-1}) \\
		& = \sum_n b_1(x_n ) \rho_n h_k(x-x_n) \\
		& = \rho_b * h_k(x) ,
	\end{split}
\end{equation}
where $\rho_b = \rho \cdot b_1$.

\section{Connection between the Generalized Model and MH3D}
Both MH3D and the generalized MPI reconstruction model introduced in \cite{sanders2025physics} rely on the following form of the received signal
\begin{equation}\label{eq: cc1}
	s_0(t) = m \iiint \rho(\vec x) \vec b_1\T(\vec x) \vec{\vec{h}}(\vec \xi (t) - \vec x) \vec \xi'(t)  \, \d \vec x,
\end{equation}
from which the discretized linear matrix model is given by
\begin{equation}
	\vec s_0 = V E H B \rho.
\end{equation}
This discretized model encapsulates each of the components from the continuous integral formulation in (\ref{eq: cc1}). Namely, $V$ accounts for the velocity component, $E$ accounts for the FFP location, $H$ accounts for $\vec{\vec h}$ as the magnetic convolutional component, and $B$ accounts for the receive sensitivity.

The MH3D model applies a harmonic filter and then a gridding operation to the received signal $s_0$. We write these two operations as $C_k$ and $\mathcal I_{t,x}$ respectively, which is expressed as

\begin{equation}
	\begin{split}
		d_k & = \mathcal I_{t,x} C_k s_0\\
		& = \mathcal I_{t,x} C_k V E H B \rho
	\end{split}
\end{equation}

On the other hand, according to (\ref{eq: matrixModel}) we have
\begin{equation}
	d_k = \mathbf{H_k} B\rho,
\end{equation}
where $\mathbf{H_k} = [H_{kx} , H_{ky} , H_{kz}]$, and $B = [B_x, B_y, B_z]\T$ is the same matrix as in (\ref{eq: cc1}). Based on this line of reasoning, we deduce that
\begin{equation}
	\mathcal I_{t,x} C_k V E H = \mathbf{H_k} .
\end{equation}
While this is not a formal proof of equivalence between the models, this paper demonstrates a striking resemblance between the two as shown in the images output in Figures \ref{fig: model} and \ref{fig: mh2d}.


\section{The Information at Each Harmonic}
Here, we provide intuition and details of the information contained in each harmonic portrait. Observe from equations (\ref{eq: hkapprox0}) and (\ref{eq: hkapprox}) that we are approximately sampling the SPION density convolved with derivatives of the Langevin function. If we were able to measure the first harmonic with infinite SNR, leaving us with $\rho$ convolved with the first derivative of the Langevin function, then this would have sufficient information to restore $\rho \ge 0$, since $\mathcal L'$ is a positive function which decays rapidly akin to a Gaussian PSF. In this case, we would only need to deconvolve this result. However, since we only observe $\rho$ convolved with 2nd and higher derivatives of the Langevin, certain objects in $\rho$ would fall into a null space, and thus be un-recoverable. For example, since 
\begin{equation}
	0 	= \int \mathcal L''(x) \, dx,
\end{equation}
then it is easy to see that any constant background in $\rho$ would be completely eliminated in the higher harmonics. Likewise, any polynomial of degree $k$ in $\rho$ is eliminated by convolving $\rho$ with $\mathcal L^{(k+2)}$. To that end, we can intuit that the lower harmonics provide the most important information to restore the fundamental components of $\rho$, while the higher harmonics contain more functions in the null space. On the other hand, the higher harmonic PSFs are sharper and improve the resolution in the realistic case where the SNR is limited (see Section \ref{sec: sampling} for details).

To proceed with this argument more formally, we first set $\myC = 1$ in the argument of the Langevin function to simplify the discussion, without any loss of generality. Let's consider $\rho$ to be a nonnegative and continuous infinitely differentiable function on some bounded interval $[a,b]$ and denote this space of functions by \begin{equation}
	\mathcal S := \{ \rho(x) \, | \, \rho(x)\ge 0 \quad \text{and} \quad \rho \in C^{\infty}[a,b] \} .
\end{equation}
Next, let's define the set of functional operators acting on $\mathcal S$ by
\begin{equation}
	\mathbb F_k( \rho   ) := \mathcal L^{(k)} * \rho(x) ,
\end{equation}
for $k\ge 1$, while noting that the range of $\mathbb F_k$ is effectively the set of functions that can be seen in the $k$th harmonic portrait domain. Then it is easy to deduce that $\text{null}(\mathbb F_1) = \varnothing$, while for $k\ge 2$, the null space of $\mathbb F_k$ contains polynomials of degree strictly less than $k-1$. So as $k$ increases, the respective null space also increases. Hence with infinite SNR, then any functions that we are able to recover from the range of $\mathbb F_k$ we are also able to recover from $\mathbb F_{k-1}$.

In summary, the lower harmonics provide the fundamental information needed to recover the image structure. The first harmonic is the only harmonic that retains the background constant of the image, so within MPI we must make assumptions about the image background, which has typically assumed that the regions near the boundaries are zero (see Section \ref{sec: alpha} for details). The second harmonic retains all other image information, and more fundamental image information is lost as the harmonic number increases. However, the higher harmonics provide higher SNR information about the function's derivatives, which improves the resolution of the images when dealing with pragmatic MPI sampling patterns, which is discussed in Section \ref{sec: sampling}.

\section{MHAD: A Native Image Reconstruction based on Portrait Data}\label{sec: MHAD}
This section describes a simple native image reconstruction based on an anti-differentiation method from harmonic portraits, which we call multi-harmonic anti-differentiation (MHAD). The images returned from this method are approximately equivalent to those obtained from the X-space stitching method, while in practice particular nuances will be observed from each approach. However, for MHAD to work in practice, the portraits need to be sampled sufficiently dense for the discretized approximations used below to be accurate. 

\begin{figure}
	\centering
	\includegraphics[width=.75\textwidth]{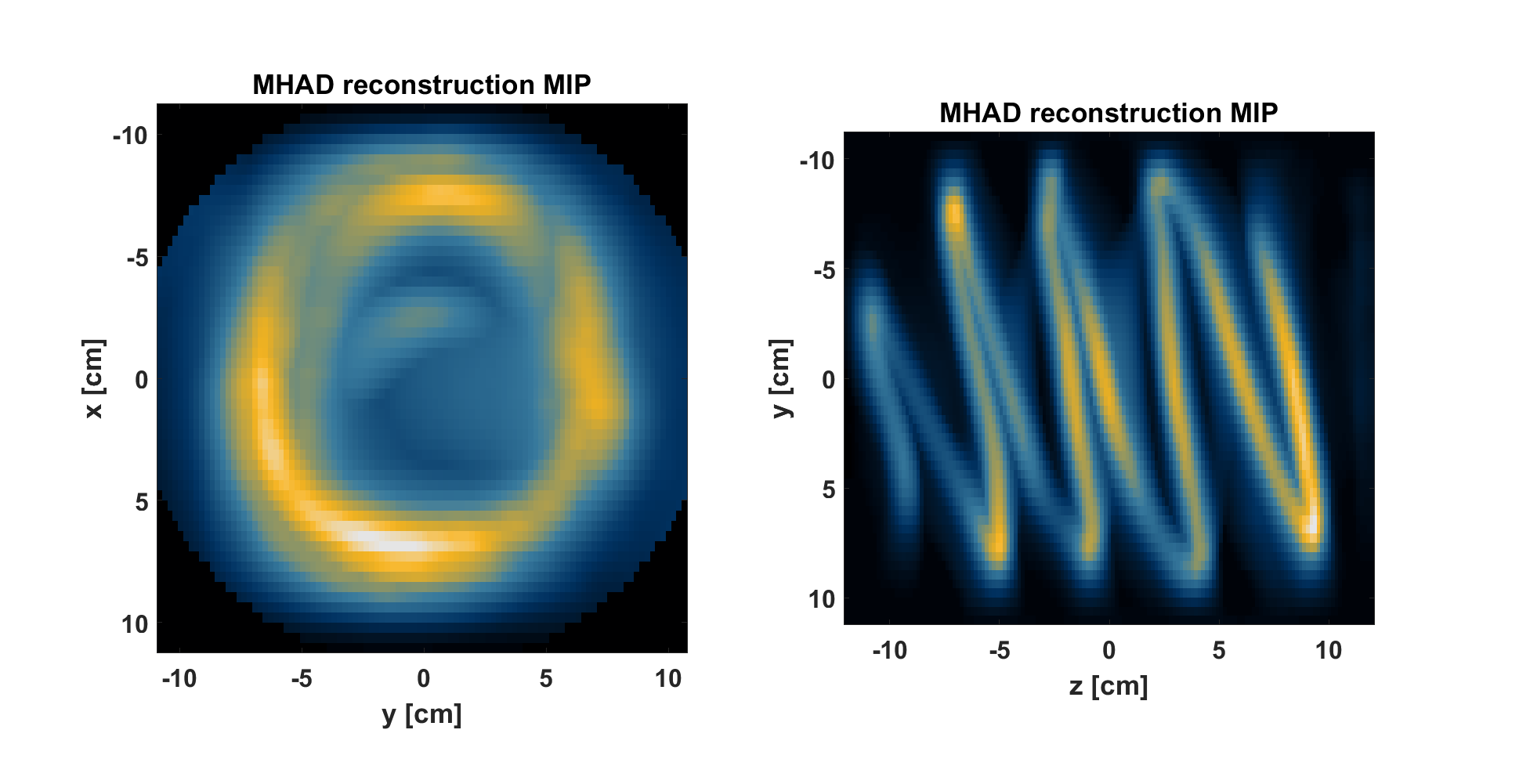}
	\caption{An example of a MHAD reconstruction using the same experimental dataset from the example shown in Figure \ref{fig: 3Dspiral}.}
	\label{fig: mhad}
\end{figure}

We continue our discussion in a 1D setting, but the same arguments apply in 3D, where we only need to perform the derived methods along the transmit axis. Equation (\ref{eq: 2AD}) shows how we may evaluate an anti-derivative with the 2nd harmonic portrait to estimate the native image, where the native image is defined by 
$$
\rho_N(x) = h * \rho(x) = \mathcal L' * \rho(x) 
$$
as described in \cite{goodwill2010x}. To make this useful in a practical setting, we need a discrete method for anti-differentiation. Let us write a discretized model for the derivative at the 2nd harmonic by
\begin{equation}
	d_2 [j] = \rho_N[j+1] - \rho_N[j-1] + \epsilon[j],
\end{equation}
where $j$ is the pixel index and $\epsilon$ is a noise term. For computational purposes, we will make this derivative circulant so that we can write it as a convolution given by
\begin{equation}
	d_2 = f * \rho_N + \epsilon,
\end{equation}
where $f \T = \begin{bmatrix}
	0,	-1 , 0 , \dots , 0,1
\end{bmatrix}.
$
Then we may write a 2nd harmonic only regularized MHAD model as
\begin{equation}\label{eq: mhad2}
	\begin{split}
		\rho_N^* & =  \arg\min_\rho  \|  f*\rho - d_2 \|_2^2 + \lambda \| \rho \|_2^2\\
		& = \mathcal F^{-1} \left[\frac{\overline{\hat f} }{ |\hat f|^2 + \lambda } \cdot \hat d_2 \right] ,
	\end{split}
\end{equation}
where $\mathcal F$ is the discrete Fourier transform and the hats denote the Fourier transforms of the vectors (see \cite{sanders2020effective} for details and examples of more interesting regularizers).

To arrive at the full multi-harmonic version, we write 
$$
f^k = \underbrace{f*f*\dots *f}_{\text{$k$ times}},
$$
and introduce relevant scaling factors for each portrait given by $c_k = \frac{(\myC A)^k}{(k-1)!}$ (see Appendix \ref{sec: analysis} for details). Then our more general portrait model at the $k$th harmonic is given as
$$
d_k = c_k \cdot f^{k-1} * \rho_N  + \epsilon_k.
$$
Then the complete MHAD solution is given by
\begin{equation}\label{eq: mhadM}
	\begin{split}
		\rho_N^* & =  \arg\min_\rho \sum_{k=2}^K \|  f^{k-1}*\rho - c_k^{-1} d_k \|_2^2 + \lambda \| \rho \|_2^2\\
		& = \mathcal F^{-1} \left[\frac{\sum_k c_k^{-1} \overline{\hat f^{k-1}} \cdot d_k  }{ \sum_k |\hat f^{k-1} |^2 + \lambda }  \right] .
	\end{split},
\end{equation}
Note that convenience of the analytical MHAD solutions written in (\ref{eq: mhad2}) and (\ref{eq: mhadM}) is that the computational method for solving them only requires a few FFTs and are therefore extremely cheap computational solutions.


\section{Proof of Theorem \ref{prop main}} \label{sec: proof}
To prove Theorem \ref{prop main} we first establish a few basic propositions that will be needed. The first is a simple and well-known result\cite{goodwill2010x}.

\begin{prop}
	Consider a 1D MPI signal model with an FFP at time $t$ given by $\xi(t)$ and a total magnetic field given by $H(x,t) = G(\xi(t) - x)$. Then the unfiltered received signal for an arbitrary SPION density $\rho(x)$ is given by
	\begin{equation}\label{eq: prop1}
		s_0(t) =m \myC \xi'(t) \cdot \rho_b * h(\xi(t)),
	\end{equation}
	where $h(x) = \mathcal L'(\myC x)$ and $\rho_b(x) = \rho(x) \cdot b_1(x)$. 
\end{prop}
\begin{proof}
	Recall the unfiltered received signal in 1D MPI is given by
	\begin{equation}
		\begin{split}
			s_0(t) 
			&= m\frac{d}{dt} \int b_1(x) \rho(x) \mathcal L\left[ \myC G^{-1} H(x,t) \right] \, dx\\
			& = m\int \rho_b(x) \frac{d}{dt} \mathcal L \left[ \myC (\xi(t) - x) \right] \, dx\\
			& = m\myC  \xi'(t) \int \rho_b(x) \mathcal L'\left[ \myC (\xi(t) - x) \right] \, dx\\
			& =m \myC  \xi'(t) \cdot  \rho_b * h(\xi(t)).
		\end{split}
	\end{equation}
\end{proof}
The following is a well-known theorem that can be found in any standard text in complex analysis. 
\begin{theorem}\label{thm comp}
	Suppose $f(z)$ is an analytic function near $z=a$  in the complex plane. Then $f(z)$ has a power series expansion of the form
	$$
	f(z) = \sum_{k=0}^\infty f^{(k)}(a)\frac{(z-a)^k}{k!},
	$$
	and the radius of convergence of this series is the largest value of $R$ such that $f(z)$ is analytic on the set $\{|z-a|<R\}$.
\end{theorem}

\begin{prop}\label{prop: taylor}
	Let $h(x) = \mathcal L'(\myC x)$ , where $\myC >0$ and $\mathcal L$ is the Langevin function given by $\mathcal L(x) = \coth (x) - 1/x$. Then a particular Taylor expansion of $h$ is given as
	\begin{equation}\label{eq: s01}
		h(a+b) = \sum_{k=0}^\infty h^{(k)}(a) \frac{b^k}{k!},
	\end{equation}
	which converges for 
	$$
	|\myC b| < \sqrt{\myC^2 a^2 + \pi^2}. 
	$$
\end{prop}

Proof of proposition \ref{prop: taylor} does not require a detailed calculation, but rather a careful analysis of the Langevin function and the corresponding convergence of its Taylor series according to Theorem \ref{thm comp}. In particular, the Langevin function (and its derivatives) have discontinuities in the complex plane at $ik\pi$ with $k$ being a nonzero integer. Then a typical Taylor series of the Langevin function written by 
$$
\mathcal L(x) = \sum_{k=0}^\infty \mathcal L^{(k)} (x) \frac{(x-a)^k}{k!},
$$
converges for values $|x| < \sqrt{a^2 + \pi^2}$. Translating these concepts to our particular function $h(x)$ completes the proposition.

Finally we have all of the ingredients necessary to prove the main result.
\begin{proof}[Proof of Theorem \ref{prop main}]
	Substituting the values for $\xi(t)$ and $\rho(x)$ into (\ref{eq: prop1}) obtains
	\begin{equation}
		\frac{s_0(t)}{m\xi'(t) b_1(x_0)} = \myC  \cdot h(A e^{i2\pi f_0 t} + t \Delta_t - x_0) .
	\end{equation}
	Applying the expansion for $h$ given in (\ref{eq: s01}) to the above equation yields
	\begin{equation}
		\begin{split}
			\frac{s_0(t)}{m\xi'(t) b_1(x_0 )} & =\myC  \cdot h(A e^{i2\pi f_0 t} + t \Delta_t - x_0)\\
			& = \myC \sum_{k=0}^\infty
			h^{(k)}(t\Delta_t - x_0) \frac{A^k e^{i2\pi f_0 kt }}{k!}  \\
			& =  \myC \sum_{k=0}^\infty \frac{A^k}{k!}
			h^{(k)}(t\Delta_t - x_0) {e^{i2\pi f_0 kt }}
		\end{split}
	\end{equation}
	Multiplying through by $m\xi'(t) b_1(x_0) $ and performing rearrangements completes the proof. 
\end{proof}

\section{Multi-dimensional Derivation of the Portrait Model with a Single-Axis Transmit}
Here we outline the details for the multi-dimensional derivation of the portrait model, which was more formally derived for the 1D case. Begin by recalling the unfiltered 3D MPI received signal is given by
\begin{equation}\label{eq: 1}
	\begin{split}
		s_0(t) & = m \iiint \rho(\vec x ) \vec b_1\T(\vec x) \frac{G(\vec \xi(t) - \vec x)}{\| G(\vec \xi(t) - \vec x) \|} \mathcal L [\beta \| G(\vec \xi(t)  - \vec x) \|] \, d\vec x
		\\
		&  = m \vec v(t)\T \vec{\vec{h}} * \vec \rho_b (\xi(t)) \\
		& = m \sum_{i,j=1}^3 v_j(t) h_{ji}*(\rho b_i) (\xi(t)) , 
	\end{split}
\end{equation}
where 
\begin{itemize}
	\item $\rho(\vec x)$ is the SPION(s) we are imaging.
	\item $\vec b_1(\vec x)$ is the receive coil sensitivity. 
	\item $\vec \xi (t)$ is the FFP.
	\item $\beta$ is a conglomerate nanoparticle constant.
	\item $m$ is the magnetic moment.
	\item $G$ is the $3\times 3$ static magnetic gradient field matrix.
	\item $\vec{\vec{h}}(\vec x)$ is the $3\times 3$ PSF tensor function given by
	
	\begin{equation}\label{eq: PSF}
		\begin{split}
			\vec{\vec{h}} (\x) := 
			\Bigg[ &\mL'(\| G \x \|/H_{sat}) \frac{G \x \x\T G\T }{\| G \x \| H_{sat}}
			+   \mL (\| G \x \|/H_{sat}) \left( I - \frac{G \x \x\T G\T }{\| G \x \|^2}
			\right) \Bigg] \frac{G}{\| G \x \|}.
		\end{split}
	\end{equation}
	\item $\vec v (t)$ is the FFP velocity. 
\end{itemize}

We proceed with assuming a scan sequence and SPION density given as
\begin{equation}
	\rho(\vec x) = \delta(\vec x - \vec x_0)
\end{equation}
and 
\begin{equation}
	\vec \xi(t) = t\cdot \vec \Delta + A e^{i2\pi f_0 t} \cdot \vec e_3 ,
\end{equation}
where $\vec \Delta$ is some linear shift rate vector and $\vec e_3 = (0, 0, 1)\T$. 

Inputting this into the last line of (\ref{eq: 1}) and then using a Taylor expansion along the $z$-axis obtains
\begin{equation}
	\begin{split}
		s_0(t) & \approx 
		im2\pi f_0  \sum_{j=1}^3 A e^{i2\pi f_0 t} h_{3j} (t \cdot \vec \Delta  + A e^{i2\pi f_0 t} \cdot  \vec e_3 - \vec x_0) b_j(\vec x_0)\\
		& = C \sum_{j=1}^3 b_j(\vec x_0) 
		\sum_{k=0}^\infty
		\frac{ h_{3j}^{(0,0,k)} ( t\cdot \vec \Delta - \vec x_0 )  }{k! } A^{k+1} e^{i2\pi f_0 (k+1) t} ,
	\end{split},
\end{equation}
where $C = im2\pi f_0$, and in the approximation we dropped the $O(\|\vec \Delta\|)$ terms coming from the FFP velocity shift vector. In what follows we write $h_{3j}^{(k)}$ to denote $h_{3j}^{(0,0,k)}$ for notational convenience.

Filtering for the $k$th harmonic then leads to
\begin{equation}\label{eq: skMD}
	\begin{split}
		s_k(t) & = \int s_0(\tau ) e^{-i2\pi f_0 k \tau } g(t - \tau) \, d \tau \\
		& \approx C \sum_{j=1}^3 b_j(\vec x_0) \frac{A^k}{(k-1)!} \int  h_{3j}^{(k-1)} (\tau \vec \Delta  - \vec x_0 ) g(t - \tau) \, d\tau .
	\end{split}
\end{equation}
We may grid this signal back into the spatial domain to the FFP focus field point to form the portrait via $t = (\vec \Delta)^{-1} \vec x $, where here it is implied that $(\vec \Delta)^{-1}\vec x$ will select any one of the shift direction which is nonzero for the inverse mapping. For example, if $\Delta_y$ being the shift rate in $y$ is nonzero, then we may define $(\vec \Delta)^{-1}\vec x := \Delta_y^{-1} y$.

Making this substitution obtains
\begin{equation}
	\begin{split}
		s_k(\vec \Delta^{-1} \vec x )
		&
		\approx 
		C \sum_{j=1}^3 b_j( x_0 ) \frac{A^k}{(k-1)!} \int  h_{3j}^{(k-1)} (\tau \vec \Delta - \vec x_0 ) g(\vec \Delta^{-1} \vec x - \tau) \, d\tau 
		\\
		& = 
		C \sum_{j=1}^3 b_j( x_0 )\underbrace{ \frac{A^k}{(k-1)!} \int  h_{3j}^{(k-1)} ((\vec x-\vec x_0) - \Delta w ) g(w ) \, dw }_{=:  (\vec h_{k})_j ( \vec x - \vec x_0 )}
	\end{split}
\end{equation} 

For a generalized SPION source written as 
$$
\rho(\vec x) = \sum_{n}\rho_n \delta(\vec x - \vec x_n) ,
$$
then (\ref{eq: skMD}) becomes
\begin{equation}
	\begin{split}
		s_k(t) & \approx C \sum_{j=1}^3 \sum_n \rho_n b_j(\vec x_n) \frac{A^k}{(k-1)!} \int  h_{3j}^{(k-1)} (\tau \vec \Delta  - \vec x_n ) g(t - \tau) \, d\tau ,
	\end{split}
\end{equation}
and the portrait is given by
\begin{equation}
	\begin{split}
		d_k(\vec x) & = s_k(\vec \Delta^{-1} \vec x ) \\
		& \approx C \sum_{j=1}^3 \sum_n \rho_n b_j(\vec x_n) 
		\frac{A^k}{(k-1)!} \int  h_{3j}^{(k-1)} (\tau \vec \Delta  - \vec x_n ) g(\vec \Delta^{-1} \vec x - \tau) \, d\tau
		\\
		& = C \sum_{j=1}^3 \sum_n \rho_n b_j(\vec x_n) 
		\underbrace{\frac{A^k}{(k-1)!} \int  h_{3j}^{(k-1)} ((\vec x - \vec x_n) - \Delta w ) g(w) \, dw}_{(\vec h_{k})_j (\vec x - \vec x_n)}\\
		& = C \sum_{j=1}^3 \sum_n  \rho_n b_j(\vec x_n)   (\vec h_{k})_j  (\vec x - \vec x_n)\\
		& =  C \cdot \vec h_k * \vec \rho_b (\vec x), 
	\end{split}
\end{equation}
where $\vec \rho_b (\vec x) = \rho(\vec x) \cdot \vec b_1 (\vec x).$

\section*{Acknowledgment}
The authors gratefully acknowledge Albert Boggess for drawing attention to the convergence of the Taylor series in the complex plane, and Olivia C. Sehl for designing the lymph node head phantom used in the experimental studies.

\bibliographystyle{unsrt}

\end{document}